\documentclass[paper=a4, english, 12pt]{amsart}

\usepackage{latexsym}
\usepackage{cite}
\usepackage{amsfonts}
\usepackage{amsmath}
\usepackage{amsthm}
\usepackage{amssymb}
\usepackage{mathrsfs}
\usepackage{mathabx}
\usepackage{bbm}           
\usepackage{dsfont}

\allowdisplaybreaks

\usepackage{wrapfig}

\usepackage{framed}

\usepackage{tikz}
\usepackage{color}
\usepackage{graphicx}

\usepackage{subfigure}
\usepackage{fancyvrb}
\usepackage{cases}
\usepackage{multicol}
\usepackage{xspace}

\usepackage[utf8]{inputenc}
\usepackage[T1]{fontenc}

\usepackage{listingsutf8}

\usepackage{textcomp}
\usepackage{lmodern}
\usepackage[a4paper,margin=2cm]{geometry}
\usepackage[english]{babel}

\usepackage{hyperref}

\newtheorem{theorem}{Theorem}[section]
\newtheorem{corollary}{Corollary}
\newtheorem{lemma}[theorem]{Lemma}
\newtheorem{proposition}{Proposition}[section]
\newtheorem{definition}[theorem]{Definition}
\newtheorem{remark}{Remark}
\newtheorem{example}{Example}

\newcommand{\eN}{\mathbbm{N}}
\newcommand{\N}{\mathbbm{N}}

\newcommand{\eR}{\mathbbm{R}}
\newcommand{\R}{\mathbbm{R}}
\newcommand{\eZ}{\mathbbm{Z}}

\newcommand{\Cp}[1]{\mathbf{C}^{#1}}
\newcommand{\Lp}[1]{\mathbf{L}^{#1}}
\newcommand{\Wp}[1]{\mathbf{W}^{#1}}
\newcommand{\Lip}{\mathbf{Lip}}

\newcommand{\Cc}[1]{\mathbf{C}_c^{#1}}

\newcommand{\BV}{\mathbf{BV}}
\newcommand{\Lploc}[1]{\mathbf{L}^{#1}_{loc}}

\DeclareMathOperator{\sign}{sign}

\newcommand{\eps}{\varepsilon}

\newcommand{\norma}[1]{{\left\|#1\right\|}}
\newcommand{\norm}[1]{{\left\|#1\right\|}}

\newcommand{\sgn}[1]{\mathrm{sign}\left(#1\right)}

\newcommand{\rie}{\rho_i^\eps}
\newcommand{\rhe}{\rho^\eps_{h}}

\newcommand{\rje}{\rho_j^\eps}

\begin{document}

\title[Well-posedness for a monotone solver for traffic junctions]{Well-posedness for a monotone solver\\ for traffic junctions}

\author[B.~Andreianov]{Boris~P.~Andreianov}
\address{Laboratoire de Math\'ematiques et Physique Th\'eorique CNRS UMR7350\\
Universit\'e de Tours\\
Parc de Grandmont\\
37200 Tours
France}
\email{Boris.Andreianov@lmpt.univ-tours.fr}

\author[G.~M.~Coclite]{Giuseppe Maria Coclite}
\address{Dipartimento di Matematica\\
Universit\`a degli Studi di Bari\\
Via Orabona 4\\
70125 Bari\\
Italy}
\email{giuseppemaria.coclite@uniba.it}

\author[C.~Donadello]{Carlotta Donadello}
\address{Laboratoire de Math\'ematiques CNRS UMR6623\\
Universit\'e de Franche-Comt\'e\\
16 route de Gray\\
25030 Besan{\c c}on Cedex\\
France}
\email{carlotta.donadello@univ-fcomte.fr}
\thanks{The second author is member of the Gruppo Nazionale per l'Analisi Matematica, la Probabilit\`a e le loro Applicazioni (GNAMPA) of the Istituto Nazionale di Alta Matematica (INdAM). \\
The work on this paper was supported by the French ANR project CoToCoLa.\\
 The third author also acknowlegdes support from the Universit\'e de Franche-Comt\'e (projet support EC 2014) and from the R{\'e}gion Franche-Comt\'e (projet mobilit\'e sortante a.a. 2015-16). 
 She would like to thank the Mathematics' Department of the University of Padova for the warm hospitality during the a.y. 2015-16. }



\maketitle

\begin{abstract}

In this paper we aim at proving well-posedness of solutions obtained as vanishing viscosity limits for the Cauchy problem on a traffic junction where $m$ incoming and $n$ outgoing roads meet. The traffic on each road is governed by a scalar conservation law $ \rho_{h,t} + f_h(\rho_h)_x  = 0$, for $h\in \{1,\ldots, m+n\}$. Our proof relies upon the complete description of the set of road-wise constant solutions and its properties, which is of some interest on its own. Then we introduce a family of Kruzhkov-type adapted entropies at the junction and state a definition of admissible solution in the same spirit as in \cite{diehl, ColomboGoatinConstraint, scontrainte, AC_transmission, germes}.
\end{abstract}

\section{Introduction}
We consider a junction consisting of $m$ incoming and $n$ outgoing roads. Incoming roads are parametrized by $x\in\eR_-$ while outgoing road by $x\in\eR_+$ in such a way that the junction is always located at  $x=0$.

We describe the evolution of traffic on each road by a scalar conservation law of the form
\begin{equation}\label{eq:basic}
  \rho_{h,t} + f_h(\rho_h)_x  = 0, \qquad\text{for } h= 1,\ldots,\, m+n,
\end{equation}
where $\rho_h$ is the density of vehicles and $f_h$ is the flux on the $h$-th road. For notational simplicity we call $\Omega_h$ the spatial domain of the density $\rho_h$. Everywhere in the paper we use the index $i$ for the $m$ incoming roads and $j$ for the $n$ outgoing roads (then $\Omega_i =\eR_-$ for all $i =1,\ldots, m$ and $\Omega_j =\eR_+$ for all $j =m+1,\ldots, m+n$), see Figure~\ref{fig:junction}.
The fluxes $f_h$, $h= 1 , \ldots, m+n$, differ in general as each road may have different maximal capacities and speed limitations. However, we assume that each flux $f_h$ is bell-shaped (unimodal), Lipschitz and non-degenerate nonlinear i.e. it satisfies the conditions
\begin{enumerate}
  \item[\textbf{(F)}]  for all $h$, $f_h \in \Lip\left( [0,R]; \eR_+ \right)$ with $\|f_h'\|_{\infty}\leq L_h$, $f_h(0) = 0 = f_h(R)$,\\
   and there exists $\bar\rho_h \in \left]0,R\right[$ such that $f_h'(\rho)~(\bar\rho_h-\rho)>0$ for a.e.~$\rho \in [0,R]$,
    \item[\textbf{(NLD)}]   for all $h$,  $f'_h$ is not constant on any non-trivial subinterval of $[0,R]$.
\
\end{enumerate}
The fundamental postulate of our approach is that any physically relevant solution of the problem has to satisfy, as minimal requirement, the conservation of the total density at the junction. The intuitive way to express this condition is to say that for a.e. $t\in\eR_+$
\begin{equation}\label{conservation}
  \sum_{i=1}^m f_i\left(\rho_i(t,0^-)\right)  = \sum_{j=m+1}^{m+n} f_j\left(\rho_j(t,0^+)\right).
\end{equation}
\begin{figure}\label{fig:junction}
  \includegraphics[width=.5\textwidth]{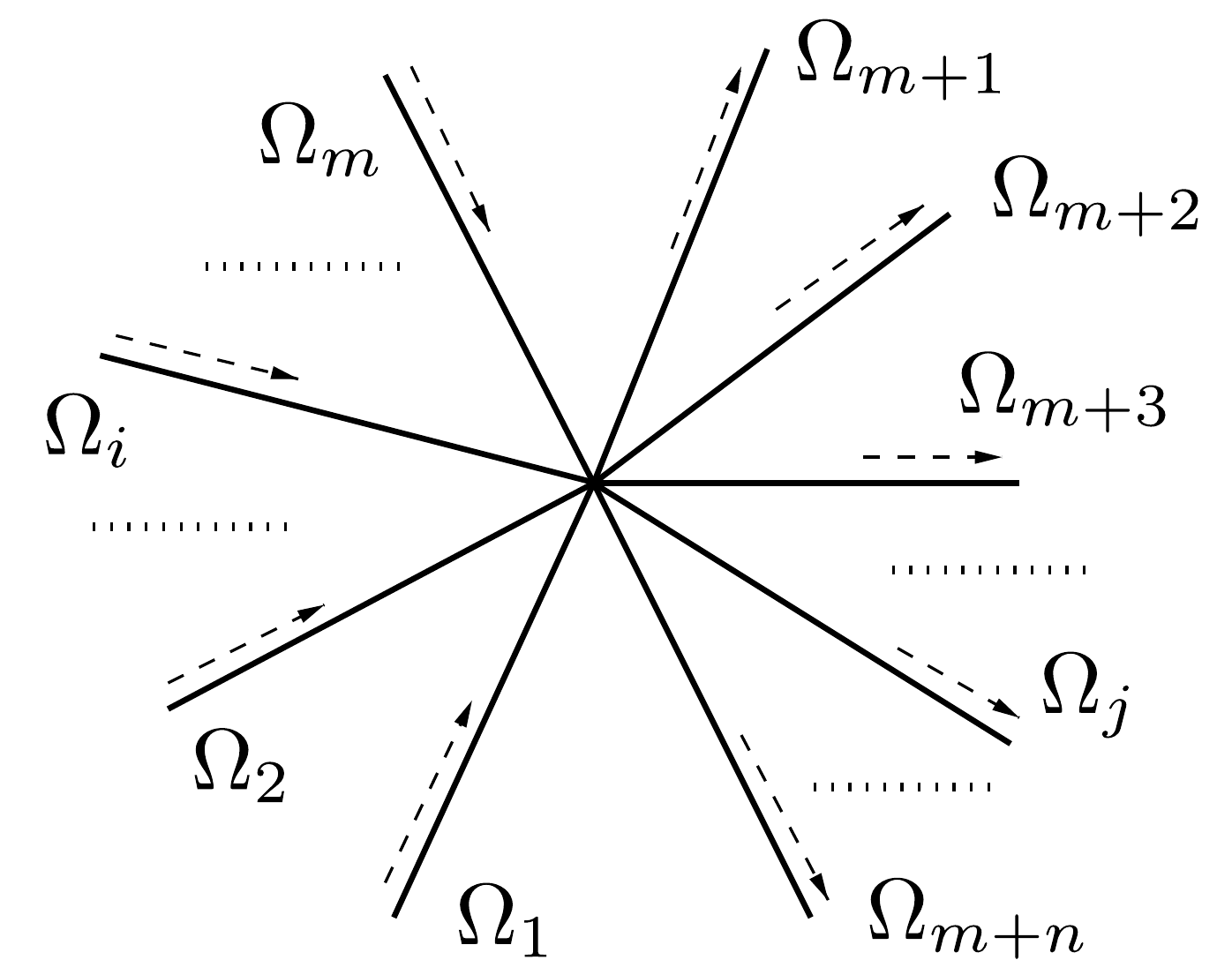}
  \caption{A junction consisting of $m$ incoming and $n$ outgoing roads.}
\end{figure}

Garavello and the second author, in \cite{coclitegaravello}, considered the Cauchy problem at the junction and established the existence of weak solutions obtained as limit of vanishing viscosity approximations. In \cite{coclitegaravellopiccoli}, uniqueness for such solutions was only proved in the special case $m=n$ and $f_h=f_{h'}$ for all $h,h'\in\{1,\ldots,m+n\}$. The present paper naturally completes those results as we obtain the uniqueness of the vanishing viscosity limit for any number of roads. Our approach relies upon a partial generalization of the recent results on scalar conservation laws with discontinuous flux obtained by the first authors and his collaborators, see in particular \cite{AC_transmission,germes}.
 Let us mention in passing that a large part of the concepts and results of \cite{AC_transmission,germes} can be generalized to conservation laws on networks. However,  a systematic generalization of the theory of $\Lp1$-dissipative germs is beyond the scope of the present paper: we focus on characterization of 
  solutions to the concrete problem \eqref{eq:basic} originating from the vanishing viscosity regularization of \cite{coclitegaravello}, and on well-posedness in this framework. Our presentation is essentially self-contained. Let us only mention that in~\cite{ImbertMonneauZidani,ImbertMonneau}, the authors provide general results, indirectly exploiting some insight from \cite{germes}, for a junction whose traffic is described by Hamilton-Jacobi equations.

\begin{remark}\label{rem:linkstoDFSCL}
  For readers acquainted with the discontinuous-flux theory, let us indicate that we characterize the admissibility of solutions at the junction in terms of the ``vanishing viscosity germ'' $\mathcal{G}_{VV}$ (cf. \cite{germes,AKR-VV}) which is introduced under the form that was put forward in \cite{AC_transmission} (Definition~\ref{definitiongerm}). Note that we give three equivalent definitions of admissible solutions, different definitions being useful for different purposes (meaning of the junction admissibility condition, proof of uniqueness, proof of existence).

We provide the interpretation of $\mathcal{G}_{VV}$ in terms of Oleinik-like inequalities of \cite{diehl} (Lemma~\ref{Oleinik}). We prove that this germ is $\Lp1$-dissipative, complete and maximal (Lemma~\ref{lem:L1D}, Lemma~\ref{complete} and Lemma~\ref{maximal}, respectively). We prove the suitable Kato inequality \eqref{eq:Kato} which leads to the $\Lp1$-contraction property of the admissible solutions (Proposition~\ref{thm:uniqueness}), stability and uniqueness.

To justify existence and the relation to the vanishing viscosity regularization of \cite{coclitegaravello}, following \cite{AudussePerthame05,BKT09,germes} we introduce a family of adapted entropies at the junction (Definition~\ref{adaptedentropysolution}). We put forward the Godunov finite volume scheme inspired by \cite{AC_transmission} and justify its convergence and existence of admissible solutions.  In addition, we link the definition of (a part of) the germ $\mathcal G_{VV}$ to the existence of vanishing viscosity profiles (Corollary~\ref{cor:profiles}) and identify the admissible solutions with vanishing viscosity limits (Theorem~\ref{VVsolareGentropysol}).
\end{remark}

A second important remark is that our uniqueness result is by no means a result on the uniqueness \emph{tout court} of solutions of the Cauchy problem on a traffic junction. It is well known in the literature that different Riemann Solvers can be used at junctions, depending on the physical situation one aims at describing, see \cite{garavello2006traffic, coclitegaravellopiccoli, HoldenRisebro95}, the recent survey \cite{BressanCanicEtAl} and references therein.
{Let us point out that the definitions and results of Section~\ref{sec:formulations} (starting from \S~\ref{ssec:Riemann}) and Section~\ref{sec:well-posed} can be adapted in a straightforward way to the study of solutions corresponding to Riemann Solvers at the junction which verify the order-preservation property (increasing the Riemann datum on any of the roads results in pointwise increase of the solution on the whole network) and the Lipschitz continuity properties of the corresponding Godunov fluxes, cf. the last paragraph of Remark~\ref{rem:junctionGodunov}. However, the order-preservation property of Riemann solvers at junctions is not satisfied by most of the models proposed in the literature.}

\subsection{Preliminaries}\label{sec:Prelim}

We assume that the reader is acquainted with the notion of entropy solution to scalar conservation laws introduced by Kruzhkov \cite{Kruzkov}. This notion
is suitable for describing admissibility of solutions to \eqref{eq:basic} away from the junction.
 But we recall, first, the formulation of the Bardos-LeRoux-N\'ed\'elec boundary condition for conservation laws in terms of
  the Godunov numerical flux, which will be instrumental for the definition of admissible solutions at the junction and for the existence proof.  Second, we recall that entropy solutions of non linearly degenerate scalar conservation laws admit boundary traces in the strong $\Lp1$ sense.

\subsubsection{Godunov's flux}\label{ssec:Godunov}
Let $u$ be  the entropy solution to the scalar conservation law with Lipschitz continuous flux
\begin{equation}\label{sclaw}
  u_t + f(u)_x =0, \qquad (t,x)\in\eR_+\times\eR
\end{equation}
corresponding to the Riemann initial condition
\begin{equation}
  u_0(x) =
  \begin{cases}
    a, & \quad \text{if } x<0, \\
    b, & \quad \text{if } x>0.
  \end{cases}
\end{equation}
One calls Godunov flux the function which associates to the couple $(a,b)$ the value $f(u(t,0^-))=f(u(t,0^+))$ (the two values are equal due to the Rankine-Hugoniot condition). The analytical expression, see for example~\cite{HoldenRisebro}, is given by
\begin{equation}
  \label{godunovflux}
G(a,b) =
\begin{cases}
  \min_{s\in[a,b]} f(s)\quad\text{if }a\leq b ,\\
  \max_{s\in[b,a]} f(s)\quad\text{if }a\geq b .
\end{cases}
\end{equation}
In the sequel, we denote by $\partial_{a} G$, resp. $\partial_b G$, the partial derivative of the Godunov flux $G$ with respect to the first, resp. to the second argument.

The Godunov flux can be used for convergent numerical approximation of \eqref{sclaw} by an explicit finite difference / finite volume scheme. This follows from the fact that $G$ satisfies the following two basic properties, shared with several other numerical fluxes as for example Rusanov and Lax-Friedrichs (see, e.g., \cite{CrandallMajda}):
\begin{itemize}
\item Consistency:  for all $a\in[0,R]$, $G(a,a) =f(a)$;
\item Monotonicity and Lipschitz continuity: There exists $L>0$ such that for all $(a,b)\in [0,R]^2$ we have
  \begin{equation}
    0\leq \partial_a G(a,b)\leq L, \qquad\qquad -L\leq \partial_b G(a,b)\leq 0.
  \end{equation}
\end{itemize}

\subsubsection{A formulation of the Bardos-LeRoux-N\'ed\'elec boundary condition}\label{ssec:BLN}
In our setting, the main interest in using Godunov flux is related to the following observation (see~\cite{DuboisLeFloch}, see also~\cite{AndreianovSbihi} for a review on this topic). Consider the initial and boundary value problem (IBVP)
\begin{equation}\label{eq:BLN-pb}
  \begin{cases}
    u_t+f(u)_x = 0,\qquad \text{for } (t,x)\:\text{in } \eR_+\times\eR_-\\
    u(t,0) = u_b (t), \\
    u(0,x) = u_0 (x),
  \end{cases}
\end{equation}
and assume that $u$ is a Kruzhkov entropy solution in the interior of the half plane $\eR_+\times\eR_-$. Then $u$ satisfies the boundary condition in the sense of Bardos-LeRoux-N\'ed\'elec (see~\cite{bardos}) if and only if its trace $\gamma u(t)= u (t, 0^-)$ satisfies $f(\gamma u(t)) = G(\gamma u(t), u_b(t))$.

\subsubsection{Strong boundary traces of local entropy solutions}\label{ssec:traces}
Consider \eqref{sclaw} locally, in $\eR_+\times (a,b)$ where $(a,b)$ is an interval of $\eR$.
Assume that $u\in \Lp\infty(\eR_+\times (a,b))$ satisfies the Kruzhkov entropy inequalities (see \cite{Kruzkov} and \eqref{entropycondition} below).
Assume that the Lipschitz flux $f$ in \eqref{sclaw} is non linearly degenerate in the sense that $f'$ is not identically zero on any interval (which follows from $(NLD)$).
Then (see~\cite{Panov}, see also \cite{germes}) the function $u(t,\cdot)$ possesses one-sided limits: e.g., one can define $u(t,b^-):=\gamma u(t)$ where
$(\gamma u)(\cdot)$ is the strong trace of $u$ on $\eR_+\times\{b\}$ in the $\Lploc1$ sense: for all $\xi\in \mathcal D(\eR_+)$,
\begin{equation}
  \label{traces-general}
  \begin{aligned}
    \lim_{k\to 0^+}\frac{1}{k}\int_{\eR_+}\int_{b-k}^b \xi(t) |u(t,x) - \gamma u(t)| \, dx\, dt=0.
    \end{aligned}
\end{equation}
Notice that this property permits to extend the above interpretation of the Bardos-LeRoux-N\'ed\'elec boundary condition for problem \eqref{eq:BLN-pb}
to the case of general $\Lp\infty$ initial and boundary data, beyond the classical $\BV$ framework.

\subsubsection{Functional framework}
Throughout the paper, we are interested in $\Lp\infty$ solutions of \eqref{eq:basic}. We will denote dy $\Gamma$ the graph pictured in Figure~\ref{fig:junction} and use the slightly abusive notation $\Lp\infty(\eR_+\times \Gamma ; [0,R]^{m+n})$ for $(m+n)$-uplets $(\rho_1,\ldots,\rho_m,\rho_{m+1},\ldots,\rho_{m+n})$ of functions such that $\rho_i\in \Lp\infty(\eR_+\times \eR_-; [0,R])$ for $i\in\{1,\ldots,m\}$ and
 $\rho_j\in \Lp\infty(\eR_+\times \eR_+ ; [0,R])$ for $j\in\{m+1,\ldots,m+n\}$. Similarly, $\Lp\infty(\Gamma ; [0,R]^{m+n})$ will denote the space of $[0,R]$-valued initial data on the graph $\Gamma$.

\subsection{The notion of admissible solution and the outline of the paper}
Our goal is to re-visit and complement the work \cite{coclitegaravello}, which studies vanishing viscosity limits for problem \eqref{eq:basic}.
The property of being a vanishing viscosity limit can be seen as a specific admissibility condition for a weak solution of \eqref{eq:basic}, which boils down to
\begin{itemize}
  \item the standard Kruzhkov entropy conditions on each of the roads $\Omega_h$, $h\in \{1,\ldots,m+n\}$;
  \item a specific ``coupling'' condition at the junction, whose description is the main object of the present paper.
\end{itemize}
An intermediate significant result of our work is the intrinsic characterization of the vanishing viscosity limits for \eqref{eq:basic}: this is done either in terms of the Riemann solver at the junction, or in terms of $m+n$ Dirichlet problems on $\Omega_h$, $h\in \{1,\ldots,m+n\}$ coupled by a simple transmission condition, or in terms od ``adapted'' entropy inequalities.

The notion of solution we aim at using is roughly speaking the following. We consider  $\vec \rho=(\rho_1, \ldots, \rho_{m+n})$ in $\Lp\infty(\eR_+\times \Gamma ; [0,R]^{m+n}) $ as an admissible solution if, first, for any $h\in\{1,\ldots, m+n\}$,  $\rho_h$ is  a weak entropy solution in the sense of Kruzhkov in the interior of $\Omega_h$. Second, recalling that $\rho_i$ (resp., $\rho_j$) admits a strong trace $\rho_i(\cdot,0^-)=\gamma_i\rho_i(\cdot)$ (resp., $\rho_j(\cdot,0^+)=\gamma_j \rho_j(\cdot)$) at $x=0$, i.e.
\begin{equation}
  \label{traces}
  \begin{aligned}
    \lim_{k\to 0^-}\frac{1}{k}\int_{\eR_+}\int_{k}^0 \xi(t) |\rho_i(t,x) - \gamma_i \rho_i(t)| \, dx\, dt=0, & \qquad \text{for } i =1,\ldots, m, \\
    \lim_{k\to 0^+}\frac{1}{k}\int_{\eR_+}\int_0^{k} \xi(t) |\rho_j(t,x) - \gamma_j \rho_j(t)| \, dx\, dt=0,&  \qquad \text{for } j =m+1,\ldots, m+n,
  \end{aligned}
\end{equation}
we require that the $(m+n)$-uple of traces satisfies condition \eqref{conservation} for a.e. $t\in\eR_+$ and, moreover,
for a.e. $t\in\eR_+$ the values of the traces ``coincide up to boundary layers''.
This choice is made in accordance with the fact that the vanishing viscosity approximation of \eqref{eq:basic} prescribes, for every viscosity parameter
$\eps>0$, the coincidence of all $\rho^\eps_h(t,\cdot)$, $h=1,\ldots,m+n$, at $x=0$; and that taking the limit $\eps\to 0$ relaxes this condition analogously to the way in which the boundary condition in \eqref{eq:BLN-pb} is relaxed.

 In order to give a more precise statement, which is the aim of this section, we need to introduce some notation.
 In Section~\ref{sec:formulations}, we will reformulate the problem in two different forms, suitable for proving the uniqueness and the existence, respectively.

\subsubsection{The junction as a collection of IBVPs}

Given an initial condition $\vec u_0 \in \Lp\infty(\Gamma ; [0,R]^{m+n})$,  $\vec u_0 = (u^0_1, \ldots, u^0_{m+n})$, we look for a function $\vec \rho=(\rho_1, \ldots, \rho_{m+n})$ in $\Lp\infty(\eR_+\times \Gamma ; [0,R]^{m+n}) $ such that for any $h\in \{1, \ldots, m+n\}$, $\rho_h$ is a weak entropy solution of the initial and boundary value problem (IBVP)
\begin{equation}\label{IBVProad}
        \begin{cases}
          \rho_{h,t} + f_h(\rho_h)_x =0, & \qquad \text{on }  ]0,T[\times\Omega_h,\\
           \rho_h(t,0) = v_h(t), & \qquad \text{on }  ]0,T[,\\
           \rho_h(0,x) = u^0_h(x), & \qquad \text{on } \Omega_h,
        \end{cases}
\end{equation}
where the set of boundary conditions $\vec v: \eR_+\to [0,R]^{m+n}$ is to be fixed in the sequel so to guarantee that, in particular, the conservativity condition \eqref{conservation} holds. Let us stress that at this point, different choices are possible, and each choice reflects a modeling assumption at
the junction.
 \begin{definition}
\label{solutionjunction}
        We say that the function $\rho_h$ is an entropy weak solution of the initial and boundary values problem \eqref{IBVProad} if
        \begin{itemize}
        \item For any test function $\xi$ in $\mathcal{D}(\eR_+\times\Omega_h ; \eR_+)$, $\xi\vert_{\partial \Omega_h} = 0$, and for any $k\in [0,R]$ there holds
            \begin{equation}
              \label{entropycondition}
              \int_{\eR_+}\int_{\Omega_h} \left\{ |\rho_h-k|\xi_t + q_h(\rho_h,k)\xi_x\right\}\,dx\,dt + \int_{ \Omega_h} |u^0_h(x)-k|\xi(0,x) \, dx\geq 0,
            \end{equation}
$q_h(u,k):=\sign(u-k)(f_h(u)-f_h(k))$ being the Kruzhkov entropy flux associated to $f_h$.
          \item For a.e. $t \in \eR_+$, $\gamma_h \rho_h(t)$ satisfies the boundary condition in the sense of Bardos-Le Roux-N\'ed\'elec (BLN), which we express under the form (cf. Section~\ref{sec:Prelim})
\begin{align}
   f_i(\gamma_i\rho_i) = G_i(\gamma_i\rho_i, v_i)&\quad\text{if } i\in \{1, \ldots, m\} ; \label{eq:God-i}\\
   f_j(\gamma_j\rho_j) = G_j(v_j, \gamma_j\rho_j)&\quad\text{if } j\in \{m+1, \ldots, m+n\} \label{eq:God-j},
\end{align}
where $G_i$ and $G_j$ are the Godunov fluxes associated to $f_i$ and $f_j$ respectively.
        \end{itemize}
      \end{definition}
 In order to describe the solutions of \eqref{eq:basic}  which can be obtained as vanishing viscosity limit, we postulate that the artificial Dirichlet values $v_h$ at the junction need to be the same for all $h$: 
 \begin{equation}\label{eq:vh-coincide}
\text{for all $h\in \{1,\ldots,m+n\}$, for a.e. $t\in\eR_+$}\;\;   v_h(t) = p(t).
 \end{equation}
We refer to \cite{AC_transmission,AM15} for detailed motivations, in the discontinuous-flux setting.
   The criterion for the choice of $p$ is the conservativity condition \eqref{conservation}; due to \eqref{eq:God-i} and \eqref{eq:God-j}, we can now express it in the form
\begin{equation}
\label{eq:defBNL}
  \sum_{i=1}^{m} G(\gamma_i\rho_i(t), p(t)) = \sum_{j=m+1}^{m+n}G(p(t), \gamma_j\rho_j(t)), \quad\text{for a.e. } t\in\eR_+.
\end{equation}
Observe that formally, \eqref{eq:vh-coincide} and \eqref{eq:defBNL} close the coupled system of IBVP's \eqref{IBVProad}.
\begin{definition}
  \label{defBLN}
Given
 an initial condition $\vec u_0\in \Lp\infty(\Gamma ; [0,R]^{m+n})$, we call $\vec \rho = (\rho_1, \ldots, \rho_{m+n})$ in $\Lp\infty \left(\eR_+\times \Gamma ; [0,R]^{m+n}\right)$ an admissible solution of \eqref{eq:basic} associated with $\vec u_0$ if there exists a function $p$ in $\Lp\infty(\eR_+; [0,R])$ such that  for any $h \in\{1,\ldots, m+n\}$ $\rho_h$ is a solution of IBVP \eqref{IBVProad} in the sense of Definition~\ref{solutionjunction} with  $v_h$, $h\in\{1,\ldots,m+n\}$ chosen to fulfill \eqref{eq:vh-coincide}, and such that $\vec \rho$, $p$ fulfill \eqref{eq:defBNL}.
\end{definition}

\subsubsection{Outline of the remaining Sections}
We will reformulate Definition~\ref{defBLN} in Section~\ref{sec:formulations}, both in terms of the  Riemann solver at the junction and in terms of adapted entropy inequalities that (unlike the ``per road'' Kruzhkov entropy inequalities \eqref{entropycondition}) account for the admissibility of $\vec \rho$ at the junction. We will establish well-posedness of problem \eqref{eq:basic} in the frame of the so defined admissible solutions in Section~\ref{sec:well-posed}. Finally, we will justify the adequacy of this definition of admissibility for intrinsic characterization of vanishing viscosity limits in Section~\ref{sec:viscosity}.

\section{Equivalent formulations of admissibility and the underlying  Riemann solver at the junction}\label{sec:formulations}
Observe that in the special case where $m=n$ and $f_h = f$ for all $h\in \{1, \ldots , 2m\}$ the constant vector function $\vec k = (k, \ldots, k)\in\eR^{2m}$ satisfies the conditions above with $p(t) = k$. This kind of stationary solution is employed in \cite{coclitegaravello} to construct a family of Kruzhkov like entropies. In general, however, other stationary solutions may
be of interest. For example in the case $m=n=1$ all vectors $\vec k = (k_1, k_2)$ such that $k_1$ and $k_2$ are respectively the left and the right state of a Kruzhkov admissible jump are admissible stationary solutions to the problem. In what follows, we introduce the vanishing viscosity germ which
 will be identified later on
 with the set of all possible stationary admissible solutions to \eqref{eq:basic} on $\eR_+\times\Gamma$ constant on each road of $\Gamma$.
 This definition will permit us to describe the Riemann solver and the associated fluxes at the junction defined in Lemma~\ref{pexists-Part2}.

\subsection{Definition of the vanishing viscosity germ}
In this section we describe the stationary admissible solutions of \eqref{eq:basic} that are constant on each road of $\Gamma$.
Because of the analogy with the discontinuous-flux setting of \cite{ AC_transmission, germes} we will use similar notation and terminology (cf. Remark~\ref{rem:linkstoDFSCL} for a brief summary).
\begin{definition}
  We call vanishing viscosity germ the subset of $[0,R]^{m+n}$ defined by
  \begin{equation}\label{definitiongerm}
    \mathcal{G}_{VV} = \left\{\begin{aligned}
    \vec u = (u_1,\ldots,u_{m+n}) : &\: \exists p\in[0,R] \:\text{such that} \\
    & \sum_{i=1}^{m} G_i(u_i, p) = \sum_{j=m+1}^{m+n}G_j(p, u_j) \\
    & \:G_i(u_i, p) = f_i(u_i), \quad G_j(p, u_j) = f_j(u_j), \:\forall i, \,j
    \end{aligned}
    \right\}.
  \end{equation}
\end{definition}
It is immediate to see that $\vec u\in \mathcal G_{VV}$ if and only if, seen as a vector function in $\eR_+\times \Gamma\to [0,R]^{m+n}$, $\vec u$ provides a solution of \eqref{eq:basic} in the sense of Definition~\ref{defBLN}
such that each component $u_h$ of $\vec u$ is constant both in $t$ and $x$.

Following \cite{germes}, see also \cite{diehl}, we can characterize $\mathcal G_{VV}$
by a set of inequalities reminiscent of the celebrated Oleinik condition for scalar conservation laws.
 Here and in the following we  use the notation $I[a,b]$ to indicate the closed interval $[\min\{a,b\}, \max\{a,b\}]$ in $\eR$.
\begin{lemma}\label{Oleinik}
 The vanishing viscosity germ $\mathcal{G}_{VV}$ coincides with the subset
 of $[0,R]^{m+n}$ defined by
  \begin{equation}\label{eq:Oleinik-like}
    \left\{\begin{aligned}
    \vec u = &(u_1,\ldots,u_{m+n}) : \: \exists p\in[0,R] \:\text{such that} \\
    & \sum_{i=1}^{m} G_i(u_i, p) = \sum_{j=m+1}^{m+n}G_j(p, u_j), \\
    & \forall s\in I[u_i,p] \quad (p-u_i)(f_i(s)-f_i(u_i))\geq 0, \:\text{for } i= 1,\ldots, m, \\
    &\forall s\in I[p,u_j] \quad (u_j-p)(f_j(s)-f_j(u_j))\geq 0, \:\text{for } j= m+1,\ldots, m+n
    \end{aligned}
    \right\}.
  \end{equation}
\end{lemma}
\begin{proof}
 Actually, the value $p$ in \eqref{eq:Oleinik-like} coincides with the value $p$ in \eqref{definitiongerm}. One only needs to show that
\[
\forall s\in I[u_i,p]\colon \quad (u_i-p)(f_i(s)-f_i(u_i))\geq 0 \Leftrightarrow G_i(u_i,p)=f_i(u_i),
\]
for any $i$. This readily comes from the definition of the Godunov flux, as the relation $G_i(u_i,p)= f_i(u_i)$ rewrites as
\begin{equation}
  \begin{aligned}
    & f_i(u_i) = \min_{s\in[u_i,p] } f_i(s), \quad\text{if } (p-u_i)\geq 0, \\
    & f_i(u_i) = \max_{s\in[p, u_i] } f_i(s), \quad\text{if } (p-u_i)\leq 0.
  \end{aligned}
\end{equation}
The proof for the $j$-index case is analogous.
\end{proof}

In order to exhibit the key properties of $\mathcal G_{VV}$, we start with the following technical lemma which is crucial for the existence theory. It relies upon the monotonicity properties of the  Godunov fluxes $G_h$, $h\in\{1,\ldots,m+n\}$; we defer to Remark~\ref{rem:junctionGodunov} for its interpretation in terms of the Godunov fluxes for the junction.
\begin{lemma}\label{pexists}
  Given $\vec u= (u_1,\ldots,u_{m+n})$ in $[0,R]^{m+n}$, consider the problem
  \begin{equation}\label{eq:pu}
   \text{find $p_{\vec u}\in[0,R]$ such that\;\;} \sum_{i=1}^{m} G_i(u_i, p_{\vec u}) = \sum_{j=m+1}^{m+n}G_j(p_{\vec u}, u_j).
  \end{equation}
  \begin{itemize}
    \item[(i)] The set $\mathcal P_{\vec u}$ of solutions of \eqref{eq:pu} is non-empty.
    \item[(ii)] The values $(G_i(u_i, p_{\vec u}))_{i\in\{1,\ldots,m\}}$, $(G_j(p_{\vec u},u_j))_{j\in\{m+1,\ldots,m+n\}}$
    do not depend on the choice of the value $p_{\vec u}\in \mathcal P_{\vec u}$.
   \end{itemize}
  \end{lemma}

  \begin{example}
   To give an example, consider a junction consisting of two incoming and one outgoing roads, on which the traffic is described through the following flux functions:
   \begin{equation}
     f_h(\rho) = -h\rho^2+h, \qquad \text{for } h = 1,2,3.
   \end{equation}
Remark that, for the reader's convenience, we consider now $\rho\in[-1,1]$, $f_h (-1) = f_h(1) = 0$ and $\bar \rho_h = 0$. This does not change our results but allows for cleaner computations.
As above we call $G_h$ the Godunov flux corresponding to $f_h$ and $u_h$ is the constant initial condition on the $h$-th road. If $u_h\neq 0$, let $\hat u_h\neq u_h$ be the only solution to $f_h(u_h) = f_h(\hat u_h)$.
By using the standard Riemann Solver, see~\cite{HoldenRisebro}, it is easy to check that the values of $G_i(u_i,\cdot)$, $i= 1,\,2$, as functions of $p$, are the following
\begin{itemize}
\item If $u_i\leq 0$, then $G_i(u_i,p) = f_i(u_i)$ for all $p \leq \hat u_i $ and $G_i(u_i,p) = f_i(p)$ for all $p \geq \hat u_i $;
\item If $u_i\geq 0$, then $G_i(u_i,p) = f_i(0)$ for all $p \leq 0 $ and $G_i(u_i,p) = f_i(p)$ for all $p \geq 0$.
\end{itemize}
Similarly, the values of $G_3(\cdot, u_3)$, as a function of $p$ are
\begin{itemize}
\item If $u_3\leq 0$, then $G_3(p, u_3) = f_3(p)$ for all $p \leq 0$ and $G_3(p, u_3) = f_3(0)$ for all $p \geq 0 $;
\item If $u_3\geq 0$, then $G_3(p, u_3) = f_3(p)$ for all $p \leq \hat u_3$ and $G_3(p, u_3) = f_3(u_3)$ for all $p \geq \hat u_3 $.
\end{itemize}
One can check that if we take, as an example, $u_1 = -\sqrt{1/2}$, $u_2= 1/4$ and $u_3 = \sqrt{1/6}$, then all the values of $p$ between $[-\sqrt{1/6}, 0]$ satisfy the relation
\begin{equation}
  G_1(u_1,p) + G_2(u_2, p) = G_3(p,u_3),
\end{equation}
and that for all these values of $p$, the collection $\vec G^*(\vec u)$ of fluxes at the junction will be given by
$$
\vec G^*(\vec u)=\Bigl( G_1(u_1,p), G_2(u_2,p), G_3(p,u_3)\Bigr)=(1/2,2,5/2).
$$
Remark that, as explained in the paper \cite{lebacque1996godunov}, the functions $G_i(u_i, \cdot)$ and $G_j(\cdot, u_j)$ are closely related to the equilibrium supply/demand functions introduced in the work by Lebacque and his collaborators. In particular we have that the equilibrium demand function $\Delta_i$ of the $i$-th incoming road  and the equilibrium supply function $\Sigma_j$ of the $j$-th outgoing road can be defined as
\begin{equation}
  u_i\mapsto \Delta_i(u_i) = \max_p \{G_i (u_i, p)\}, \qquad u_j\mapsto \Sigma_j (u_j) = \max_p \{G_j(p, u_j)\}.
\end{equation}
  \end{example}

Now, we prove Lemma~\ref{pexists}.
\begin{proof}
~\\
(i)  Given $\vec u$ we define the functions $\Phi^{in}_{\vec u}$ and $\Phi^{out}_{\vec u}$ from $[0,R]$ to $\eR$ by
  \begin{equation}
    \label{phis}
\Phi^{in}_{\vec u}:p\mapsto \sum_{i=1}^{m} G_i(u_i, p), \qquad \Phi^{out}_{\vec u}:p\mapsto \sum_{j=m+1}^{m+n}G_j(p, u_j).
  \end{equation}
A quick direct calculation gives us
\begin{equation}
  \begin{aligned}
G_i(u_i, 0) = \max_{s\in[0,u_i]} f_i(s)\geq 0 , &\quad G_i(u_i, R) =\min_{s\in[u_i,R]} f_i(s)=0,\\
  G_j(0,u_j) = \min_{s\in[0,u_j]} f_j(s) =0 , &\quad G_j(R,u_j) =\max_{s\in[u_j,R]} f_i(s)\geq 0,
  \end{aligned}
\end{equation}
so that
\begin{equation}
  \Phi^{in}_{\vec u}(0)\geq 0=\Phi^{out}_{\vec u}(0) \quad\text{and } \Phi^{in}_{\vec u}(R)=0\leq \Phi^{out}_{\vec u}(R).
\end{equation}
The existence of at least one solution $p_{\vec u}$ of \eqref{eq:pu} is ensured by the continuity of $\Phi^{in}_{\vec u} - \Phi^{out}_{\vec u}$ on $[0,R]$ and the intermediate value theorem.

\smallskip
\noindent
(ii) It may happen that there exist several values of $p$ such that $\Phi^{in}_{\vec u}(p)$ and $\Phi^{out}_{\vec u}(p)$ coincide.
From the Lipschitz continuity and monotonicity property of Godunov flux (see \S~\ref{ssec:Godunov}) we have
\begin{equation}\label{monotonePhi}
\forall i\in\{1,\ldots,m\}\;\; \partial_b G_i(u_i, p) \leq 0,\;\;
\forall j\in\{m+1,\ldots,m+n\}\;\; \partial_a G_j(p, u_j) \geq 0,
\end{equation}
which means in particular that $\Phi^{in}_{\vec u} - \Phi^{out}_{\vec u}$ is non-strictly decreasing. Therefore, the set $\mathcal P_{\vec u}$ of solutions of \eqref{eq:pu} is a closed sub-interval of $[0,R]$. Since, moreover, each term of the sums defining $\Phi^{in}_{\vec u}-\Phi^{out}_{\vec u}$ has the same monotonicity,
we find that all these terms are constant on $\mathcal P_{\vec u}$.
\end{proof}

Next, consider the map $\vec G^*:[0,R]^{m+n}\mapsto \eR^{m+n}$ which is well defined, thanks to (ii):
   \begin{equation}\label{eq:junctionsolver-def}
\vec G^*(\vec u)=\Bigl(G_1^*(\vec u),\ldots,G_{m+n}^*(\vec u)\Bigr),\;\;   \begin{array}{l}
     G^*_i(\vec u):=G_i(u_i, p_{\vec u}),\; i\in\{1,\ldots,m\},\\[2pt]
     G^*_j(\vec u):=G_j(p_{\vec u},u_j),\; j\in\{m+1,\ldots,m+n\}
   \end{array}
   \end{equation}
   and the map $F^*:[0,R]^{m+n}\mapsto \eR$ defined by
   \begin{equation}\label{defF}
  F^*(\vec u):=\sum_{i=1}^{m} G^*_i(\vec u)\equiv \sum_{j=m+1}^{m+n}G^*_j(\vec u).
\end{equation}
\begin{lemma}\label{pexists-Part2}
With the above definitions, the following properties hold.
  \begin{itemize}
    \item[(i)] For each $i\in\{1,\ldots,m\}$, the map $\vec u\mapsto G^*_i(\vec u)$ fulfills
    $$
    \partial_{u_i} G^*_i \leq L_i\;\;\;\text{and}\;\;\; \forall h\in\{1,\ldots,m+n\},\, h\neq i\colon\;\; \partial_{u_h} G^*_i \leq 0.
    $$
     For each $j\in\{m+1,\ldots,m+n\}$, the map $\vec u\mapsto G^*_j(\vec u)$ fulfills
     $$
  \partial_{u_j} G^*_j \geq - L_j  \;\;\;\text{and}\;\;\; \forall h\in\{1,\ldots,m+n\},\, h\neq j\colon \;\; \partial_{u_h} G^*_j \geq 0.
    $$
    \item[(ii)] The map $\vec u\mapsto F^*(\vec u)$ fulfills
    $$
    \forall i\in\{1,\ldots,m\}\colon\;\; \partial_{u_i} F^* \geq 0 \;\;\;\text{and}\;\;\;
    \forall j\in\{m+1,\ldots,m+n\}\colon\;\; \partial_{u_j} F^* \leq 0.
    $$
  \end{itemize}
\end{lemma}
The above differential inequalities should be understood in the sense of distributions, e.g., ``$\partial_{u_i} F^* \geq 0$'' means that $F^*$ is non-decreasing in the variable $u_i$.
\begin{proof}~\\
(i) Without loss of generality, we can fix the normalization $p_{\vec u}:=\min \mathcal P_{\vec u}$. Observe that
\begin{equation}\label{eq:p(u)-monotonicity}
  \text{the map $\vec u\mapsto p_{\vec u}$ is monotone non-decreasing in each component $u_h$, $h\in\{1,\ldots,m+n\}$}.
\end{equation}
Indeed, let $\vec u \leq \vec v$ in the sense of the per component order: for all $h\in\{1,\ldots,m+n\}$ $u_h\leq v_h$. Bearing in mind formula \eqref{phis}, the monotonicity properties of the Godunov fluxes $G_i(\cdot,p)$ and $G_j(p,\cdot)$ and the definition of $p_{\vec u}$ yield
$$
0=\Phi^{in}_{\vec u}(p_{\vec u}) - \Phi^{out}_{\vec u}(p_{\vec u})\leq  \Phi^{in}_{\vec v}(p_{\vec u}) - \Phi^{out}_{\vec v}(p_{\vec u}).
$$
Hence the monotonicity of $\Phi^{in}_{\vec v} - \Phi^{out}_{\vec v}$ exhibited in the proof of Lemma~\ref{pexists}(ii) and the normalization of $p_{\vec u}$ ensure that $p_{\vec u}\leq p_{\vec v}$, which proves \eqref{eq:p(u)-monotonicity}.

Now the monotonicity claims of (i) are immediate from the monotonicity properties of $G_h$ and from \eqref{eq:p(u)-monotonicity}. To prove the one-sided Lipschitz continuity properties of  $G^*_h$ claimed in (i), let us focus on $h=1$. The other cases are proved in the same way. Given $\vec u=(u_1,u_2,\ldots,u_{m+n})$ and $\vec v=(v_1,u_2,\ldots,u_{m+n})$ with $v_1>u_1$, we have
\begin{align*}
 G^*_1(\vec v)-G^*_1(\vec u) & = G_1(v_1,p_{\vec v})- G_1(u_1,p_{\vec u})\\
 & = G_1(v_1,p_{\vec v})- G_1(u_1,p_{\vec v}) + G_1(u_1,p_{\vec v})- G_1(u_1,p_{\vec u})\\
 & \leq G_1(v_1,p_{\vec v})- G_1(u_1,p_{\vec v}) \leq L_1(v_1-u_1),
\end{align*}
which proves the claim.

\noindent
(ii) The monotonicity properties of $F^*$ readily stem from \eqref{eq:p(u)-monotonicity}. If we look at the dependence of $F^*$ in $u_i$, $i\in\{1,\ldots,m\}$, it is enough to represent $F^*$ with the last expression in \eqref{defF} and combine \eqref{eq:p(u)-monotonicity} with the monotonicity of $G_j(\cdot,b)$, $j\in \{m+1,\ldots,m+n\}$. If we look at the dependence of $F^*$ in $u_j$, $j\in\{m+1,\ldots,m+n\}$, then we  represent $F^*$ with the first expression in \eqref{defF} and use  the monotonicity of $G_i(a,\cdot)$, $i\in \{1,\ldots,m\}$. 
\end{proof}

\begin{remark}
  Actually, in the context of Lemma~\ref{pexists-Part2}(i) we can also prove that $G^*_i$ (resp., $G^*_j$) is monotone non-decreasing (resp., non-increasing) in the argument $u_i$ (resp., $u_j$), and therefore it is $L_i$ (resp., $L_j$) Lipschitz continuous. Indeed, it is enough to represent, e.g.,
  $G^*_1$ by the following expression derived from \eqref{eq:pu}:
  $$
  G^*_1(p_{\vec u})=-\sum_{i=2}^m G_i(u_i,p_{\vec u}) + \sum_{j=m+1}^{m+n} G_j(p_{\vec u},u_j),
  $$
  and exploit \eqref{eq:p(u)-monotonicity}.
However, we do not stress these finer properties of dependence of $G^*_i$ on $u_i$ because they are not essential for the subsequent analysis.
\end{remark}

\smallskip
Now, we are ready to explore the crucial ``dissipativity'' properties of $\mathcal G_{VV}$.
For any $h\in \{1,\ldots, m+n\}$ let $q_h: [0,R]^2\to \eR$ denote the Kruzhkov entropy flux
\begin{equation}
  q_h(u,v) = \sign(u-v)(f_h(u)-f_h(v)).
\end{equation}
\begin{lemma}\label{lem:L1D}
  For any $\vec k_1 $, $\vec k_2$ in $\mathcal{G}_{VV}$  with $\vec k_\ell = (k_1^\ell, \ldots, k_{m+n}^\ell)$, $\ell = 1, \, 2$, there holds
  \begin{equation}\label{toward_kato}
  \Delta(\vec k_1,\vec k_2):= \sum_{i=1}^{m} q_i(k_i^1,k_i^2) - \sum_{j=m+1}^{m+n}q_j(k_j^1,k_j^2)\geq 0.
  \end{equation}
\end{lemma}
 Inequality \eqref{toward_kato} will be exploited in this work to prove that the solutions we consider satisfy a generalized Kato's inequality. The $\Lp1$-contraction property and uniqueness will follow.
 \begin{proof}
To increase readability we use the shorter notation $s_h= \sign(k_h^1-k_h^2)$, for $h = 1, \ldots, m+n$, and we adopt the convention $\sign(0)=0$. Up to reordering the incoming and outgoing roads we can assume that $s_i \geq 0 $ for $i\in\{1,\ldots, \alpha\}$, $s_i=-1$ for $i=\alpha+1$,  $s_i\leq 0$ for $i\in\{\alpha+1,\ldots, m\}$, $s_j \geq 0 $ for $j\in\{m+1,\ldots, \beta\}$, $s_j=-1$ for $i=\beta+1$ and $s_i\leq 0$ for $i\in\{\beta+1,\ldots, m+n\}$.

Intuitively we deal with three cases. Actually, the third case is more general and includes the first two.
   \begin{description}
     \item[Case 1] Assume that $s_i s_j\geq 0$ for all $(i,j)\in\{1,\ldots, m\}\times \{m+1,\ldots, m+n\}$.
\item[Case 2] Assume that $s_i s_j\leq 0$ for all $(i,j)\in\{1,\ldots, m\}\times \{m+1,\ldots, m+n\}$.
\item[Case 3] Let $\alpha\neq 1,\, m$ and $\beta\neq m+1, \,  m+n$.
   \end{description}

\vspace{0.2cm}

\noindent \textbf{Case 1} This is the easiest case. The definition of $\mathcal{G}_{VV}$ implies that
\begin{equation}
  \begin{aligned}
&    \sum_{i=1}^{m} q_i(k_i^1,k_i^2) - \sum_{j=m+1}^{m+n}q_j(k_j^1,k_j^2) = \pm\sum_{i=1}^{m} \left(f_i(k_i^1) - f_i(k_i^2)\right) \mp \sum_{j=m+1}^{m+n} \left(f_j(k_j^1) - f_j(k_j^2)\right) \\
&= \pm\sum_{i=1}^{m}G_i(k_i^1, p_1) \mp \sum_{j=m+1}^{m+n}G_j(p_1, k_j^1) \mp\sum_{i=1}^{m}G_i(k_i^2, p_2) \pm \sum_{j=m+1}^{m+n}G_j(p_2, k_j^2) =0,
  \end{aligned}
\end{equation}
where we call $p_1$ and $p_2$ the values of $p$ associated respectively to $\vec k_1 $ and $\vec k_2$.

\vspace{0.2cm}

\noindent \textbf{Case 2}
\begin{equation}
  \begin{aligned}
&    \sum_{i=1}^{m} q_i(k_i^1,k_i^2) - \sum_{j=m+1}^{m+n}q_j(k_j^1,k_j^2) = \pm\sum_{i=1}^{m} \left(f_i(k_i^1) - f_i(k_i^2)\right) \pm \sum_{j=m+1}^{m+n} \left(f_j(k_j^1) - f_j(k_j^2)\right) \\
& =\pm2\sum_{i=1}^{m}G_i(k_i^1, p_1) \mp 2\sum_{i=1}^{m}G_i(k_i^2, p_2).
  \end{aligned}
\end{equation}
Assume, to fix the ideas, that $s_i>0$, and $s_j<0$. Then the expression above writes as
\begin{equation}
 2 F^*(\vec k^1) - 2F^*(\vec k^2)
\end{equation}
with the notation \eqref{defF}, and this expression is for sure non negative thanks to the monotonicity properties of the function $F^*$ established in Lemma~\ref{pexists-Part2}(ii).

\vspace{0.2cm}

\noindent \textbf{Case 3} In this case we define the vector $\vec w $, whose components satisfy $w_h = \min\{k_h^1,\,k_h^2\}$. It is important to notice that the vector $\vec w$ does not belong (in general) to $\mathcal{G}_{VV}$.
From the proof of Lemma~\ref{pexists-Part2} (see \eqref{eq:p(u)-monotonicity}) we deduce that $p_{\vec w}\leq \min\{p_1,p_2\}$ where we denoted $p_s:=p_{\vec k^s}$, $s=1,2$. We have, using the notation \eqref{phis} and the fact that $w_i=k_i^1$ for $i\in\{1,\ldots,m\}$ and $w_j=k_j^2$ for $j\in\{m+1,\ldots,m+n\}$,
\begin{equation*}
  \begin{aligned}
   \sum_{i=1}^{m} q_i(k_i^1,k_i^2) = & \sum_{i=1}^{m}s_i \left(G_i(k_i^1, p_1) - G_i(k_i^2, p_2)\right) \\
 =&\sum_{i=1}^{\alpha}s_i \left(G_i(k_i^1, p_1) - G_i(w_i, p_2)\right) + \sum_{i=\alpha+1}^{m}s_i \left(G_i(w_i, p_1) - G_i(k_i^2, p_2)\right)\\
\geq& \sum_{i=1}^{\alpha}s_i \left(G_i(k_i^1, p_1) - G_i(w_i, p_{\vec w})\right) + \sum_{i=\alpha+1}^{m}s_i \left(G_i(w_i, p_{\vec w}) - G_i(k_i^2, p_2)\right)\\
=& \sum_{i=1}^{m} \sign(k_i^1- w_i) \left(G_i(k_i^1, p_1) - G_i(w_i, p_{\vec w})\right) \\
&+\sum_{i=1}^{m} \sign(k_i^2- w_i) \left(G_i(k_i^2, p_2) - G_i(w_i, p_{\vec w})\right)\\
=& \Phi^{in}_{\vec k^1}(p_1) - \Phi^{in}_{\vec w}(p_{\vec w}) - \sum_{i=\alpha+1}^{m} \left(G_i(k_i^1, p_1) - G_i(k_i^2, p_{\vec w})\right)\\
& + \Phi^{in}_{\vec k^2}(p_2) - \Phi^{in}_{\vec w}(p_{\vec w}) - \sum_{i=1}^{\alpha} \left(G_i(k_i^2, p_2) - G_i(k_i^2, p_{\vec w})\right).
  \end{aligned}
\end{equation*}
 Analogous computations on the sum involving the outgoing roads give a similar result. Therefore we obtain
 \begin{equation*}
  \begin{aligned}
&\sum_{i=1}^{m} q_i(k_i^1,k_i^2) - \sum_{j=m+1}^{m+n} q_j(k_j^1,k_j^2)  \\
&\geq - \sum_{i=\alpha+1}^{m} \left(G_i(k_i^1, p_1) - G_i(k_i^i, p_w)\right) - \sum_{i=1}^{\alpha} \left(G_i(k_i^2, p_2) - G_i(k_i^2, p_w)\right)\\
&\quad+ \sum_{j=\beta+1}^{m+n} \left(G_j(k_j^1, p_1) - G_i(k_j^1, p_w)\right) + \sum_{j=m+1}^{\beta} \left(G_j(k_j^2, p_2) - G_j(k_j^2, p_w)\right).
  \end{aligned}
\end{equation*}
Each of the four summands in the right hand side is positive, therefore the desired inequality holds.
 \end{proof}

\subsection{Riemann problem at the junction}\label{ssec:Riemann} 
In this section we discuss the Riemann solver at the junction associated to the vanishing viscosity limit, i.e. to the Definition~\ref{defBLN} of admissible solution. A very general definition of Riemann solver at junctions is provided in \cite[Def.~2.3]{BressanCanicEtAl}. In the framework of the present paper, such definition specializes as follows
\begin{definition}\label{RS}
  The Riemann solver associated to the vanishing viscosity limit is a function
  \begin{equation}
    \mathcal{RS} : [0,R]^{m+n} \to[0,R]^{m+n}, \qquad \mathcal{RS}(\vec u_0) = \vec u,
  \end{equation}
with the following properties:
\begin{enumerate}
\item  There exists $p\in [0,R]$ such that
\begin{align}
  & f_i(u_i) = G_i(u^0_i, p)\quad\text{if } i\in \{1, \ldots, m\} ; \\
  & f_j(u_j) = G_j(p, u^0_j)\quad\text{if } j\in \{m+1, \ldots, m+n\},
\end{align}
where $G_i$ and $G_j$ are the Godunov fluxes associated to $f_i$ and $f_j$ respectively, and  $\sum_{i=1}^{m} f_i(u_i) = \sum_{j=m+1}^{m+n} f_j(u_j)$.
\item The consistency condition $\mathcal{RS}\left(\mathcal{RS}(\vec u_0)\right) = \mathcal{RS}(\vec u_0)$ holds for all $\vec u_0$ in $[0,R]^{m+n}$.
\end{enumerate}
\end{definition}
It is clear from the definition above that $\mathcal G_{VV}$ coincides with the set of \emph{equilibria}, see\cite[Def.\,2.5]{BressanCanicEtAl}, for the Riemann Solver obtained as vanishing viscosity limit. The next lemma claim that a Riemann solver in the sense of Definition~\ref{RS} exists.

\begin{lemma}\label{complete}
  For any given initial condition $\vec u_0$ in $[0,R]^{m+n}$ there exists a self-similar function $\vec \rho = (\rho_1, \ldots, \rho_{m+n})$ in $\Lp\infty \left(\eR_+\times\Gamma ; [0,R]^{m+n}\right)$  which is an admissible entropy solution of the Riemann problem at the junction in the sense of Definition~\ref{defBLN}. In particular, this means that the vector $\vec{\gamma\rho}$ of traces of this solution at the junction belongs to $\mathcal{G}_{VV}$.
\end{lemma}
\begin{proof}
  Given $\vec u_0$ in $[0,R]^{m+n}$ we apply Lemma~\ref{pexists} and find $p:=p_{\vec u_0}\in[0,R]$ such that $\sum_{i=1}^m G_i(u_i^0, p) = \sum_{j=m+1}^{m+n} G_j(p, u_j^0) $. Then we consider the $h$ initial boundary value problems with constant data
  \begin{equation}\label{IBVPcd}
    \begin{cases}
      \rho_{h,t} + f_h(\rho_h)_x =0, &\qquad \text{on }  ]0,T[\times\Omega_h,\\
\rho_h(t,0) = p, &\qquad \text{on }  ]0,T[,\\
\rho_h(0,x) = u^0_h,  &\qquad \text{on } \Omega_h.
    \end{cases}
  \end{equation}
Call $\rho_h$ the Kruzhkov entropy weak solution to \eqref{IBVPcd} and $\gamma_h \rho_h$ its (strong) trace at $\partial\Omega_h$, satisfying the boundary condition in the sense of Bardos-Le Roux-N\'ed\'elec.
Because the solution is unique and the problem is invariant under the scaling $(t,x)\mapsto (ct,cx)$ for all $c>0$, the solution is self-similar, i.e., each of the components $\rho_h$ depends only on the ratio $\frac xt$.
To conclude the proof, it is enough to observe that
\begin{equation}\label{BCBLN-modified}
  \begin{aligned}
    G_h(\gamma_h \rho_h, p)&= f_h(\gamma_h \rho_h) = G_h(u_h^0, p), \quad\text{if } h\leq m, \\
    G_h(p, \gamma_h \rho_h)&= f_h(\gamma_h \rho_h) = G_h(p, u_h^0), \quad\text{otherwise},
  \end{aligned}
\end{equation}
because in this case $\vec{\gamma\rho}$ fulfills the definition of $\mathcal{G}_{VV}$ with $p=p_{\vec u_0}$ and consequently, one sees that $\vec{\rho}$ is an admissible solution of the Riemann problem at the junction.
Equalities \eqref{BCBLN-modified} follow from the observations of \cite{DuboisLeFloch}. For the sake of completeness,
let's point out that, e.g., for all $i\in\{1,\ldots,m\}$,
\begin{equation}\label{BCBLN-states}
\gamma_i \rho_i =\left\{  \begin{aligned}
   & \text{argmin}_{[u^0_i,p]} f_i , \quad\text{if } u^0_i\leq p, \\
   & \text{argmax}_{[p,u^0_i]} f_i, \quad\text{if } u^0_i\geq p.
  \end{aligned}
  \right.
\end{equation}
 With this value $\gamma_i \rho_i$, the classical Riemann problem
with endstates $u^0_i$ and $\gamma_i \rho_i$ is solved only with the waves of negative speed while the
classical Riemann problem
with endstates  $\gamma_i \rho_i$ and $p$ is solved only with the waves of positive speed, so that the classical Riemann problem
with endstates $u^0_i$ and $p$ is solved by juxtaposition of the two. The definition of the Godunov flux in \S~\ref{ssec:Godunov} ensures that the values $f_i(\gamma_i \rho_i)$, $G_i(u^0_i,\gamma_i \rho_i)$, $G_i(\gamma_i \rho_i,p)$ and $G_i(u^0_i,p)$ coincide.
\end{proof}
\begin{remark}\label{rem:junctionGodunov}
Given $\vec u\in [0,R]^{m+n}$, let $p_{\vec u}$ be defined by \eqref{eq:pu}.
According to the above proof, the self-similar admissible solution $\vec{\rho}$ of the Riemann problem fulfills, with the notation \eqref{eq:junctionsolver-def},
\begin{align*}
f_i(\gamma_i\rho_i(t)) = G_i(u_i, p_{\vec u}) = G^*_i(\vec u) & \;\;\text{for all}\; i\in\{1, \ldots, m\},\\
f_j(\gamma_j\rho_j(t)) = G_j(p_{\vec u}, u_j) = G^*_j(\vec u) & \;\; \text{for all}\; j\in\{1+m, \ldots, m+n\}
\end{align*}
(recall that while $p_{\vec u}$ may not be unique, the above flux values are uniquely defined). We see that given $\vec u$,
 the collection of values $\vec G^*(\vec u)\in \eR^{m+n}$  defines, road per road, the Godunov flux at the junction associated with the Riemann solver at the junction described by Lemma~\ref{complete} (the flux is outgoing from $\Omega_i$ for $i\in\{1,\ldots,m\}$ and incoming into $\Omega_j$ for $j\in\{m+1,\ldots,m+n\}$).
 
 Lemma~\ref{pexists-Part2}(i) states that each flux $G^*_h$ is  one-sided Lipschitz with respect to $u_h$ and monotone with respect to $u_{h'}$ for all $h'\neq h$.
 Observe that the last equality in \eqref{defF} expresses the conservation property at junction.
 These properties will permit us to formulate a monotone, in the sense of \cite{CrandallMajda}, conservative finite volume scheme for approximation of admissible solutions of \eqref{eq:basic}.
\end{remark}

\begin{lemma}\label{maximal}
  Let $\vec u = (u_1, \ldots, u_{m+n})$ in $[0,R]^{m+n}$ be such that  the following family of inequalities holds
  \begin{equation}\label{eq:dual-def}
   \forall \vec k = (k_1, \ldots, k_{m+n})\in \mathcal{G}_{VV}\colon\;\;\; \Delta(\vec u,\vec k)=\sum_{i=1}^{m} q_i(u_i,k_i) - \!\sum_{j=m+1}^{m+n}q_j(u_j,k_j)\geq 0.
  \end{equation}
Then $\vec u$ is in $\mathcal{G}_{VV}$.
Moreover, being understood that the fluxes $f_h$, $h\in\{1,\ldots,m+n\}$, fulfill the condition \textbf{(F)}, the conclusion ``$\vec u\in \mathcal G_{VV}$''  still holds if in \eqref{eq:dual-def}, the condition ``$\vec k\in \mathcal G_{VV}$'' is replaced by the condition
``$\vec k\in \mathcal G_{VV}^o$'', where $\mathcal G_{VV}^o$ is the subset of $\mathcal G_{VV}$ described by
\begin{equation}\label{eq:GVV-o}
   \mathcal{G}_{VV}^o =
    \left\{\begin{aligned}
    \vec k = &(k_1,\ldots,k_{m+n}) \in   \mathcal{G}_{VV} : \: \exists p\in[0,R] \:\text{such that} \\
    & \sum_{i=1}^{m} G_i(k_i, p) = \sum_{j=m+1}^{m+n}G_j(p, k_j), \\
    & \forall s\in I[k_i,p]\setminus\{k_i\} \quad (p-k_i)(f_i(s)-f_i(k_i))> 0, \:\text{for } i= 1,\ldots, m, \\
    &\forall s\in I[p,k_j]\setminus\{k_j\} \quad (k_j-p)(f_j(s)-f_j(k_j))> 0, \:\text{for } j= m+1,\ldots, m+n
    \end{aligned}
    \right\}.
  \end{equation}
\end{lemma}
\begin{proof}
   First, let us prove the result under the assumption \eqref{eq:dual-def}. Take $\vec u = (u_1, \ldots, u_{m+n})$ as initial condition for a Riemann problem at the junction. Then consider the associated solution  $\vec v = (v_1, \ldots, v_{m+n})$   and the traces $\vec{\gamma v}$, as in the proof on Lemma~\ref{complete}. We know that $\vec{\gamma v}$ is in $\mathcal{G}_{VV}$, and therefore by the assumption \eqref{eq:dual-def},
\begin{equation}
    \sum_{i=1}^{m} q_i(u_i,\gamma_iv_i) - \sum_{j=m+1}^{m+n}q_j(u_j,\gamma_jv_j)\geq 0.
  \end{equation}
By construction we have (see the proof of Lemma~\ref{complete}) $f_h(\gamma_hv_h) = G_h(u_h, p)=G_h(\gamma_hv_h, p)$ if $h\leq m$, $f_h(\gamma_hv_h) = G_h(p, u_h)=G_h(p, \gamma_hv_h)$ if $h> m$. Moreover, by maximum principle, $\gamma_hv_h$ in between $u_h$ and $p$. This means that
\begin{equation}
  f_i(u_i) - f_i(\gamma_iv_i) =f_i(u_i) - G_i(u_i, p) =
  \begin{cases}
    f_i(u_i) -\min_{s\in[u_i, p]} f_i(s)\geq 0 &\quad \text{if } u_i\leq p, \\
    f_i(u_i) -\max_{s\in[p, u_i]} f_i(s)\leq 0 &\quad \text{if } u_i\geq p,
  \end{cases}
\end{equation}
then $f_i(u_i) - f_i(\gamma_iv_i)$ is non negative when $u_i\leq \gamma_iv_i$ and non positive when $u_i\geq \gamma_iv_i$. This means that
\begin{equation}
    \sum_{i=1}^{m} q_i(u_i,\gamma_iv_i) = - \sum_{i=1}^{m}\vert f_i(u_i) - f_i(\gamma_iv_i)\vert.
  \end{equation}
In the same way we show that
\begin{equation}
  \sum_{j=m+1}^{m+n}q_j(u_j,\gamma_jv_j) = \sum_{j=m+1}^{m+n} \vert f_j(u_j)-f_j(\gamma_jv_j)\vert.
\end{equation}
The sum of non positive terms can be non negative only if all the terms vanish. Therefore for any $h\in\{1,\ldots,m+n\}$ we have  $f_h(u_h) = f_h(\gamma_hv_h) =G_h(u_h, p)$.
In view of \eqref{definitiongerm}, this shows that $\vec v$ belongs to $\mathcal{G}_{VV}$.

Now, let us prove the last claim of the lemma. It is easily seen from the comparison of \eqref{eq:Oleinik-like} and \eqref{eq:GVV-o} that due to assumption \textbf{(F)} on the shape of the fluxes, the difference between the subsets $\mathcal G_{VV}$ and $\mathcal G_{VV}^o$ of $[0,R]^{m+n}$ consists in $m+n$-uplets $\vec k$ for which at least one of the following $(m+n)$ events occurs:
\begin{itemize}
  \item[$(A_i)$] $f_i(k_i)=f_i(p)$ and $k_i< p$;
  \item[$(B_j)$] $f_j(k_j)=f_j(p)$ and $k_j> p$.
\end{itemize}
Indeed, if, for instance, there holds
\begin{equation}\label{eq:A1-case}
\forall s\in I[k_1,p]\colon \; (p-u_1)(f_1(s)-f_1(k_1))\geq 0 \;\;\text{and}\;\; \exists s_0\in I[k_1,p]\setminus\{k_1\}\;\,\text{s.t.}\;\, f_1(s_0)=f_1(k_1),
\end{equation}
then the shape assumption $\textbf{(F)}$ tells us that $s_0=p$ and, moreover, $k_1<p$.

For the sake of being definite, assume that among $(A_i)_{i\in\{1,\ldots,m\}}$,$(B_j)_{j\in\{m+1,\ldots,m+n\}}$ the only event that occurs is $(A_1)$, namely, $f_1(k_1)=f_1(p)$, $k_1< p$ but neither $(A_i)$, $i\in\{2,\ldots,m\}$ nor $(B_j)$, $j\in\{m+1,\ldots,m+n\}$ occur.

Observe that in this case, the $(m+n)$-uplet ${\vec k}':=(p,k_2,\ldots,k_{m+n})$ belongs to $\mathcal G_{VV}^o$. Indeed, it belongs to $\mathcal G_{VV}$ since it corresponds to the same value $p$, while the event \eqref{eq:A1-case} does not occur any more since for $k_1'=p$, $I[k_1',p]\setminus\{k_1'\}$ is empty.
Moreover, whatever be $u_1\in [0,R]$, we have
$$q_1(u_1,k_1)=\sign(u_1- k_1)(f_1(u_1)-f_1(k_1)) \geq \sign(u_1- p)(f_1(u_1)-f_1( p)) =q_1(u_1, p).$$
 Consequently,
$\Delta(\vec u,\vec k) \geq
\Delta(\vec u,\vec k')\geq 0,
$
where we have used the assumption of the last claim of the lemma and the fact that ${\vec k}'\in\mathcal G_{VV}^o$.

The general case is fully analogous,
so that we find $\Delta(\vec u,\vec k)\geq 0$ not only for $\vec k\in \mathcal G_{VV}^o$, but for all $\vec k\in \mathcal G_{VV}$.
 We are reduced to the first claim of the lemma. This ends the proof.
\end{proof}

\subsection{Reformulation of admissibility in terms of traces at the junction}
Now, we are ready to give an alternative formulation of admissibility for \eqref{eq:basic}.
To this end, recall (see Section~\ref{sec:Prelim}) that local Kruzhkov entropy solutions admit boundary traces in the strong $\Lp1$ sense. Therefore the following definition makes sense.
\begin{definition}\label{Gentropysolution}
Given $m+n$ fluxes $f_h$ satisfying \textbf{(F)} and an initial condition $\vec u_0$ in $[0,R]^{m+n}$, we call $\vec \rho = (\rho_1, \ldots, \rho_{m+n})$ in $\Lp\infty \left(\eR_+\times \Gamma ; [0,R]^{m+n}\right)$ a $\mathcal{G}_{VV}$-entropy solution of \eqref{eq:basic} associated with $\vec u_0$ if
\begin{itemize}
\item the first item of Definition~\ref{solutionjunction} holds, i.e., the Kruzhkov entropy inequalities \eqref{entropycondition} hold for any $h \in\{1,\ldots, m+n\}$;
\item for a.e. $t$  in $\eR_+$,the vector $\vec{\gamma \rho }(t):=(\gamma_1 \rho_1(t),\ldots,\gamma_{m+n}\rho_{m+n}(t))$ of traces at the junction belongs to $\mathcal{G}_{VV}$.
\end{itemize}
\end{definition}
This formulation will lead to the uniqueness proof.
Before turning to the proof of equivalence of Definitions~\ref{defBLN} and~\ref{Gentropysolution}, we propose another reformulation, which will be useful for proving existence of solutions.

\subsection{Adapted entropies and another reformulation of admissibility}
Recall that, except very special cases like $n=m$ with $f_h = f$ for all $h\in\{1,\ldots, 2m\}$,
we cannot expect that constants (seen as $\vec k=(k,\ldots,k)$, $k\in [0,R]$) be solutions of \eqref{eq:basic}. Moreover, the above analysis provides us with
a wide set of stationary, constant per road solutions associated to $\vec k\in \mathcal G_{VV}$. Therefore, following \cite{BaitiJenssen,AudussePerthame05}
it is natural to express global (including junction) admissibility in terms of \emph{adapted} entropy inequalities, where the constants of the Kruzhkov-like formulation proposed in \cite{coclitegaravello} are replaced by the stationary solutions associated with states in $\mathcal G_{VV}$. This leads to the following definition.
\begin{definition}
\label{adaptedentropysolution}
Given $m+n$ fluxes $f_h$ satisfying \textbf{(F)} and an initial condition $\vec u_0$ in $[0,R]^{m+n}$, we call $\vec \rho = (\rho_1, \ldots, \rho_{m+n})$ in $\Lp\infty \left(\eR_+\times \Gamma ; [0,R]^{m+n}\right)$ an adapted entropy solution of \eqref{eq:basic} associated with $\vec u_0$ if
\begin{itemize}
\item The first item of Definitions~\ref{solutionjunction},~\ref{Gentropysolution} holds, i.e., the Kruzhkov entropy inequalities \eqref{entropycondition} hold for any $h \in\{1,\ldots, m+n\}$.
\item For any $\vec k \in \mathcal{G}_{VV}$ (which should be seen as a road-wise constant solution to the Riemann problem at the junction), $\vec\rho$ satisfies the \emph{adapted entropy inequality} on the network,
     namely,
 for any non negative test function $\xi\in\mathcal{D}(]0,+\infty[\times\eR)$
 \begin{equation}
 \label{eq:adp}
     \sum_{h=1}^{m+n}\left(\int_{\eR_+}\int_{\Omega_h} \left\{|\rho_h - k_h|\xi_{t} +q_h(\rho_h ,k_h)\xi_{x}\right\}\,dx\,dt \right)
     \geq 0.
 \end{equation}
\end{itemize}
\end{definition}

\begin{remark}
  We can consider as a first example the case in which $m=n=1$  and $f_1 = f_2 = f$. This example is not totally trivial, as the family of adapted entropies we consider is \emph{a priori} larger than the family of standard Kruzhkov entropies, because the latter correspond to the case $\vec k = (k,k)$ and $\mathcal{G}_{VV}$ contains, in addition, all vectors $\vec k = (k_1, k_2)$ such that $k_1$ and $k_2$ are respectively the left and the right state of a stationary Kruzhkov admissible jump.  Analogously, in the case $m=n$ and $f_h = f$ for all $h\in\{1,\ldots, 2m\}$ our definition seems more restrictive than the one given in the paper \cite{coclitegaravello}. In fact, the two approaches give the same result, for reasons similar to those that permit to reduce
  the second claim of Lemma~\ref{maximal} to its first claim.
\end{remark}

\begin{remark}
  In the references devoted to the theory of conservation laws with discontinuous flux (case $n=m=1$), see in particular \cite{germes, AC_transmission}, the adapted entropy inequality is sometimes written in a differential form, which is equivalent to the integral form \eqref{eq:adp} for junctions with one incoming and one outgoing road.
   The integral form \eqref{eq:adp} is the appropriate expression of adapted entropy inequalities in the case of general junctions.
\end{remark}

\subsection{Equivalence of the three formulations of admissibility}
As we already announced, the three definitions of solution admissibility actually describe one and the same notion.
\begin{theorem}\label{thm:TFAE}
  Definitions \ref{defBLN}, \ref{Gentropysolution} and \ref{adaptedentropysolution} are equivalent.
  Moreover, in Definition~\ref{adaptedentropysolution}, the set of adapted entropies can be restricted
  to $\vec k\in \mathcal G_{VV}^o$, without changing the resulting notion of solution.
\end{theorem}
\begin{proof}
We only need to establish equivalence between
the second items of the three definitions for functions $\vec \rho$ satisfying the Kruzhkov inequalities \eqref{entropycondition}
(which is a common condition for all of them).

\noindent$\bullet$ We prove that Definition \ref{adaptedentropysolution}
(even in the weaker version, where the choice of $\vec k$
for adapted entropy inequalities is restricted to $\mathcal G_{VV}^o$) 
implies Definition \ref{Gentropysolution}:

Let $\vec{\rho}$ be a solution in the sense of Definition \ref{adaptedentropysolution}.
We consider a non negative test function $\xi\in\mathcal{D}(\eR_+)$  and, for every $\alpha>0$,
$\chi_\alpha\in \Cc\infty(\R)$, such that
\begin{equation}\label{eq:chialpha}
0\le\chi_\alpha\le 1,\qquad |\chi_\alpha'|\le \frac{1}{\alpha},\qquad \chi_\alpha(x)=\begin{cases}
1,&\text{if $|x|\le\alpha$},\\ 0,&\text{if $|x|\ge 2\alpha$}.
\end{cases}
\end{equation}
With the test function $\xi\chi_\alpha$, inequality \eqref{eq:adp} becomes
 \begin{equation}
   \begin{aligned}
&     \sum_{h=1}^{m+n} \int_{\eR_+}\int_{\Omega_h} \left\{|\rho_h - k_h|\xi'(t)\chi_\alpha(x) +q_h(\rho_h ,k_h)\xi(t)\chi_\alpha'(x)\right\}\,dx\,dt  \\
&+\sum_{h=1}^{m+n} \int_{\Omega_h} |u^0_h(x) - k_h|\xi(0)\chi_\alpha(x)\,dx  \geq 0.
   \end{aligned}
 \end{equation}
 As $\alpha\to0$ we get
\begin{equation}
 - \int_{\eR_+} \left(\sum_{i=1}^{m} q_i(\gamma^i\rho_i ,k_i)-\sum_{j=m+1}^{m+n}q_j(\gamma^j\rho_j ,k_j)\right)\xi(t)\,dt\ge0.
 \end{equation}
Localization then yields, for a.e. $t\in\eR_+$, $\Delta(\vec{\gamma \rho}(t),\vec k)\geq 0$, in the notation of \eqref{eq:dual-def}. By Lemma~\ref{maximal},
\begin{equation}
\vec{\gamma \rho}(t)=(\gamma_1\rho_1(t),\ldots,\gamma_{m+n}\rho_{m+n}(t))\in\mathcal{G}_{VV},\qquad \text{a.e. $t\ge0$},
\end{equation}
and then $\vec{\rho}$ is a solution in the sense of Definition \ref{Gentropysolution}.

Moreover, if we require that adapted entropy inequalities of Definition~\ref{adaptedentropysolution} hold only with $\vec k\in \mathcal G_{VV}^o$,
we still get the same conclusion due to the last claim of Lemma~\ref{maximal}. This point justifies the last claim of the theorem.

\noindent$\bullet$ We prove that Definition \ref{Gentropysolution} 
implies Definition \ref{adaptedentropysolution}:

Let $\vec{\rho}$ be a solution in the sense of Definition \ref{Gentropysolution}. For every non negative test function $\xi\in\mathcal{D}(]0,+\infty[\times\eR)$, we get
\begin{equation}
 \begin{aligned}
&     \sum_{h=1}^{m+n} \int_{\eR_+}\int_{\Omega_h} \left\{|\rho_h - k_h|\xi_{t} +\left(q_h(\rho_h ,k_h)\right)\xi_{x}\right\}\,dx\,dt \\
&     -\int_{\eR_+} \biggr( \sum_{i=1}^{m}q_i(\gamma_i\rho_i (t),k_i) -\sum_{j=1}^{m+n} q_j(\gamma_j\rho_j (t),k_j) \biggr)\xi(t,0)\,dt
 \geq 0.
   \end{aligned}
 \end{equation}
Since
\begin{equation}
\vec{\gamma \rho}(t)=(\gamma_1\rho_1(t),\ldots,\gamma_{m+n}\rho_{m+n}(t))\in\mathcal{G}_{VV},\qquad \text{a.e. $t\ge0$},
\end{equation}
we have
\begin{equation}
\sum_{i=1}^{m}q_i(\gamma_i\rho_i (t),k_i) \ge\sum_{j=1}^{m+n} q_j(\gamma_j\rho_j (t),k_j) ,\qquad \text{a.e. $t\ge0$}.
\end{equation}
Therefore, $\vec{\rho}$ is a solution in the sense of Definition \ref{adaptedentropysolution}.

\noindent$\bullet$ We prove that Definition \ref{Gentropysolution}
implies Definition \ref{defBLN}:

Let $\vec{\rho}$ be a solution in the sense of Definition \ref{Gentropysolution}.
Since
\begin{equation}
\vec{\gamma \rho}(t)=(\gamma_1\rho_1(t),\ldots,\gamma_{m+n}\rho_{m+n}(t))\in\mathcal{G}_{VV},\qquad \text{a.e. $t\ge0$},
\end{equation}
there exists $p\in \Lp\infty(\R_+,[0,R])$ such that \eqref{eq:defBNL} holds for a.e. $t\ge0$. Let us point out that $p$ is indeed measurable. This results, first, from the measurability of the trace vector $\vec{\gamma \rho}:\eR_+\to [0,R]^{m+n}$; and second, given $\vec{\gamma \rho}(t)$, from the definition of $p(t)$ by equation \eqref{eq:defBNL} for which we can systematically take the smallest solution (cf. the proof of Lemma~\ref{pexists-Part2}).
Then $\vec{\rho}$ is a solution in the sense of Definition \ref{defBLN}.

\noindent$\bullet$ Finally, Definition \ref{defBLN}
implies Definition \ref{Gentropysolution}.

This follows from the definition of $\mathcal{G}_{VV}$ in \eqref{definitiongerm}, in view of \eqref{eq:God-i},\eqref{eq:God-j}.
\end{proof}


\section{Well-posedness of \eqref{eq:basic} in the frame of admissible solutions}~\label{sec:well-posed}
The goal of this section is to prove the following result.
\begin{theorem}\label{th:well-posed}
For any given initial condition $\vec u_0 = (u^0_1, \ldots, u^0_{m+n})$ in $\Lp\infty(\Gamma; \eR^{m+n})$ the problem \eqref{eq:basic} admits one and only one solution  $\rho$ in  $\Lp\infty(\eR_+\times\Gamma ; [0,R]^{m+n})$ in the sense of the equivalent Definitions~\ref{defBLN},~\ref{Gentropysolution},~\ref{adaptedentropysolution}.

Moreover, such solutions depend continuously on the initial data in the localized $\Lp1$ sense. If  $\vec \rho$ and $\vec {\hat \rho}$ are the admissible solutions corresponding respectively to the initial conditions $\vec u_0$ and $\vec v_0$, then for all $M>0$ and $t<M/L$, where $L=\max \{ \|f_h'\|_{\Lp\infty([0,R];\eR)}\,|\, h=1,\ldots,m+n\}$,
 \begin{equation}
  \begin{aligned}
   & \sum_{i=1}^{m} \int_{-(M-Lt)}^0 |\rho_{i}(t,x) - \hat \rho_{i}(t,x)|\,dx + \sum_{j=m+1}^{m+n} \int_0^{M-Lt} |\rho_{j}(t,x) - \hat \rho_{j}(t,x)|\,dx \\
        & \qquad\qquad \leq   \sum_{i=1}^{m} \int_{-M}^0 |u^0_{i}(x) - v^0_{i}(x)|\,dx + \sum_{j=m+1}^{m+n} \int_0^M |u^0_{j}(x) - v^0_{j}(x)|\,dx. \label{eq:L1loc-contraction}
  \end{aligned}
\end{equation}
In particular, the map that associates to $\vec{u}_0$ the unique corresponding admissible profile $\vec{\rho}(t)$, is non-expansive w.r.t. the $\Lp1$ distance for all $t>0$.
\end{theorem}

We justify the uniqueness of a solution admissible in the sense of Definitions~\ref{defBLN},~\ref{Gentropysolution},~\ref{adaptedentropysolution} using
Definition~\ref{Gentropysolution}. The existence of an admissible solution (which is also justified in Section~\ref{sec:viscosity} using the more technical vanishing viscosity method) is proved using Definition~\ref{adaptedentropysolution}, on the basis of a straightforward finite volume approximation with Godunov fluxes (including the Godunov fluxes at the junction discussed in Remark~\ref{rem:junctionGodunov}).

\subsection{Uniqueness of admissible solutions of \eqref{eq:basic}}
The goal of this section is to prove the following stability result
\begin{proposition}\label{thm:uniqueness}
    For any given initial condition $\vec u_0 = (u^0_1, \ldots, u^0_{m+n})$ in $\Lp\infty(
     \Gamma; \eR^{m+n})$ there exists at most one $\mathcal{G}_{VV}$-entropy solution $\rho$ in  $\Lp\infty(\eR_+\times
     \Gamma ; [0,R]^{m+n})$.

    Moreover, if  $\vec \rho$ and $\vec {\hat \rho}$ are the $\mathcal{G}_{VV}$-entropy solutions corresponding respectively to the initial conditions $\vec u_0$ and $\vec v_0$, \eqref{eq:L1loc-contraction} holds and in particular, the $\Lp1$-contraction estimate holds whenever the right-hand side is finite:
    \begin{equation}\label{adaptedcontraction}
      \begin{aligned}
        \sum_{i=1}^{m} \norma{\rho_{i}(t) - \hat \rho_{i}(t)}_{\Lp1(\eR_- ; \eR)} &+ \sum_{j=m+1}^{m+n} \norma{\rho_{j}(t) - \hat \rho_{j}(t)}_{\Lp1(\eR_+ ; \eR)}\\
        & \leq   \sum_{i=1}^{m} \norma{u^0_{i} - v^0_{i}}_{\Lp1(\eR_- ; \eR)} + \sum_{j=m+1}^{m+n} \norma{u^0_{j} - v^0_{j}}_{\Lp1(\eR_+ ; \eR)}.
      \end{aligned}
    \end{equation}
  \end{proposition}
\begin{proof}
  The first part of the proof is devoted to establish a Kato's type inequality for $\mathcal{G}_{VV}$-entropy solutions.

Let $\xi$ be a test function in $\mathcal{D}(\eR_+\times\eR ; \eR_+)$ and define, for $k>0$,
\begin{equation}\label{eq:truncation-test}
  \xi_k(t,x) = \xi(t,x)\min\left\{1,\frac{(|x| - k)^+}{k}\right\}.
\end{equation}
Then, as $\rho_h$ and $\hat \rho_h$ are Kruzhkov entropy weak solutions in the interior of $\Omega_h$, by a standard doubling of variable argument we get
\begin{equation}\label{presquekato}
  \sum_{h=1}^{m+n}\left(-\int_{\eR_+}\!\!\int_{\Omega_h} \!\!\left\{|\rho_h - \hat \rho_h|\xi_{k,t} +q_h(\rho_h ,\hat \rho_h)\xi_{k,x}\right\}\,dx\,dt - \int_{\Omega_h} \!\!|u^0_h(x) - v^0_h(x)|\xi_k(0, x)\,dx \right)\leq 0.
\end{equation}
An explicit computation shows that
\begin{equation}
  \xi_{k,x}(t,x) = \xi_x (t,x)\min\left\{1,\frac{(|x| - k)^+}{k}\right\} + \frac{1}{k}\xi(t,x) (\mathds{1}_{]k, 2k[}(x)- \mathds{1}_{]-2k, -k[}(x) ),
\end{equation}
then \eqref{presquekato} rewrites as
\begin{equation}\label{kkato}
  \begin{aligned}
    &\sum_{h=1}^{m+n}\left(-\int_{\eR_+}\int_{\Omega_h} \left\{|\rho_h - \hat \rho_h|\xi_{k,t} +q_h(\rho_h ,\hat \rho_h)\xi_{x}\min\left\{1,\frac{(|x| - k)^+}{k}\right\}\right\}\,dx\,dt\right) \\
&+\sum_{h=1}^{m+n}\frac{1}{k}\left(\int_{\eR_+}\int_{\Omega_h\cap]-2k, -k[}\!\! q_h(\rho_h ,\hat \rho_h)\xi \,dx\,dt-\int_{\eR_+}\int_{\Omega_h\cap]k, 2k[}\!\! q_h(\rho_h ,\hat \rho_h)\xi \,dx\,dt\right)\\
&+\sum_{h=1}^{m+n}\left(- \int_{\Omega_h} |u^0_h(x) - v^0_h(x)|\xi_k(0, x)\,dx\right)
\leq 0.
  \end{aligned}
\end{equation}
We have that
\begin{equation}
  \begin{aligned}
    \lim_{k\to 0^+} \sum_{h=1}^{m+n}\frac{1}{k}\left(\int_{\eR_+}\int_{\Omega_h\cap]-2k, -k[}\!\! q_h(\rho_h ,\hat \rho_h)\xi \,dx\,dt-\int_{\eR_+}\int_{\Omega_h\cap]k, 2k[}\!\! q_h(\rho_h ,\hat \rho_h)\xi \,dx\,dt\right)\\
    =\int_{\eR_+} \left(\sum_{i=1}^{m} q_i(\gamma^i\rho_i ,\gamma^i\hat \rho_i)-\sum_{j=m+1}^{m+n}q_j(\gamma^j\rho_j ,\gamma^j\hat \rho_j)\right)\xi(t,0)\,dt,
  \end{aligned}
\end{equation}
and this term is positive due to the positivity assumption on $\xi$ and the fact that $\vec{\gamma\rho}$ and $\vec{\gamma\hat \rho}$ are in $\mathcal{G}_{VV}$.
 Therefore by taking the limit as $k$ tends to $0$ of the whole left hand side in \eqref{kkato} one can conclude that
 \begin{equation}\label{eq:Kato}
    \sum_{h=1}^{m+n}\left(-\int_{\eR_+}\int_{\Omega_h} \left\{|\rho_h - \hat \rho_h|\xi_{t} +q_h(\rho_h ,\hat \rho_h)\xi_{x}\right\}\,dx\,dt - \int_{\Omega_h} |u^0_h(x) - v^0_h(x)|\xi(0, x)\,dx\right)
\leq 0,
 \end{equation}
which is Kato's inequality adapted to our setting.

The conclusion of the proof is classical, following \cite{Kruzkov}. One takes $\xi(t,x)$ approximating the characteristic function of the trapezoid
$\{(t,x)\in\eR_+\times\eR;\; |x|\leq M+Lt\}$ where $L=\max_{h\in\{1,\ldots,m+n\}} \max_{u\in[0,R]} |f_h'(u)|$ is the appropriate Lipschitz constant and $M$ is a nonnegative parameter. For a.e. fixed $t\in \eR_+$, this yields for $M>Lt$ the inequality \eqref{eq:L1loc-contraction}.

As $M\to +\infty$, we find \eqref{adaptedcontraction} as soon as its right-hand side is finite.
\end{proof}
\subsection{A finite volume numerical scheme and existence of an admissible solution}

In this section we describe a finite volume numerical scheme for the junction based on Godunov fluxes. Discretizing a fixed initial datum, we prove convergence of the discrete solutions to the unique admissible solution. We stress that there exist other numerical schemes which can be used in practice to approximate admissible solutions of \eqref{eq:basic}: we refer to \cite{AC_transmission} for the case $m=n=1$. The choice of Godunov's fluxes is motivated by the fact that this scheme is well-balanced, i.e., all admissible stationary solutions are scheme's exact solutions. The proofs of our theoretical results are easier in this setting. For other schemes based on monotone numerical fluxes (e.g., Rusanov flux), convergence to an admissible solution can also be proved, but the stationary solutions should replaced by numerical profiles (cf. Section~\ref{sec:viscosity} where analogous viscous profiles are constructed).

We fix a space step $\Delta x$. For $\ell\in \eZ$ and $h\in\{1,\ldots,m+n\}$, set $x^h_\ell:=\ell\Delta x$. We consider the uniform spatial mesh on each $\Omega_h$
\begin{align}
  &\bigcup_{\ell\leq -1} (x^i_\ell,x^i_{\ell+1})  \qquad\text{on } \Omega_i, \quad \text{for } i\leq m, \\
 &\bigcup_{\ell\geq 0} (x^j_\ell,x^j_{\ell+1})  \qquad\text{on } \Omega_j, \quad \text{for } j\geq m+1,
\end{align}
 so that the position of the junction $x=0$ on the road $\Omega_h$ corresponds to $x^h_0$.
 Then we fix a time step $\Delta t$ satisfying the CFL condition
 \begin{equation}\label{eq:CFL}
   \Delta t \leq \frac{\max_h\{L_h\}}{2}\Delta x,
 \end{equation}
 where $L_h$ is the Lipschitz constant of $f_h$.

At all cell interfaces $x^h_\ell$ for $\ell\neq 0$ we consider the standard Godunov flux $G_h$ corresponding to the flux $f_h$. At the junction $x_0^h$ we take on each road $\Omega_h$ the Godunov flux $G_h^*$ corresponding to the admissible solution of the Riemann problem at the junction, see Remark~\ref{rem:junctionGodunov}. More precisely, let $\vec{u_0}= (u^0_1,\ldots, u^0_{m+n})$ be an initial condition in $\Lp\infty(\Gamma; [0,R]^{m+n})$. We initialize our scheme by discretizing the initial conditions
\begin{equation}\label{eq:ICscheme}
  u_{k+\frac{1}{2}}^{h, 0} = \frac{1}{\Delta x} \int_{x^h_\ell}^{x^h_{\ell+1}} u^0_h(x)\, dx ,
\end{equation}
for all $h = 1, \ldots , m+n$ and for $\ell\leq -1$ if $h\leq m$, $\ell\geq 0$ if $h\geq m+1$.

Then $\left( u_{\ell+\frac{1}{2}}^{h, s+1} \right)_{h,\ell}$ is obtained from $\left( u_{\ell+\frac{1}{2}}^{h, s} \right)_{h,\ell}$ through the following scheme
\begin{enumerate}
\item[(I)] We solve
\begin{equation}\label{eq:p-step-scheme}
 \text{find $p^s$ in $[0,R]$ s.t.}\;\; \sum_{i=1}^{m}G_i(u_{-\frac{1}{2}}^{i, s}, p^s) - \sum_{j=m+1}^{m+n}G_j(p^s, u_{\frac{1}{2}}^{j, s})=0,
\end{equation}
i.e., setting
\begin{equation}\label{eq:vec u star}
  \vec u^{\,*,s}:=\Bigl(u_{-\frac{1}{2}}^{1, s},\ldots,u_{-\frac{1}{2}}^{m, s},u_{\frac{1}{2}}^{m+1, s},u_{\frac{1}{2}}^{m+n, s}  \Bigr)
\end{equation}
we take $p^s:=p_{\vec u^{\,*,s}}$ according to the notation of Lemma~\ref{pexists}.
\item[(II)] We compute
  \begin{equation}\label{eq:main-scheme}
    u_{\ell+\frac{1}{2}}^{h, s+1} =  u_{\ell+\frac{1}{2}}^{h, s} -\frac{\Delta t}{\Delta x} \left( \mathcal{F}^{h, s}_{\ell+1}  - \mathcal{F}^{h, s}_{\ell} \right),
  \end{equation}
where
\begin{equation}\label{eq:schemefluxes}
  \mathcal{F}^{h,s}_{\ell}  =
  \begin{cases}
    G_h(u_{\ell-\frac{1}{2}}^{h, s},u_{\ell+\frac{1}{2}}^{h, s}) \quad&\quad\text{if } h\leq m \:\text{ and }\ell\leq -1  \:\text{ or } h\geq m+1 \:\text{ and }\ell\geq 1, \\
    G_h(u_{-\frac{1}{2}}^{h, s}, p^s)\equiv G^*_h(\vec u^{\,*,s}) &\quad\text{if } h\leq m \:\text{ and }\ell=0,  \\
    G_h(p^s,u_{\frac{1}{2}}^{h, s})\equiv G^*_h(\vec u^{\,*,s}) &\quad\text{if } h\geq m+1 \:\text{ and }\ell=0.
  \end{cases}
\end{equation}
\end{enumerate}
The stage (II) is a standard marching scheme, up to the specific definition of the fluxes for $\ell=0$.  The stage (I) is implicit: once per time step, we have to find a zero of a scalar nonlinear function. Moreover, the nonlinear term in (I) is monotone and continuous
(see Lemma~\ref{pexists} and its proof) but it is not everywhere differentiable, since $G_h$ are not everywhere differentiable. In practice,
the value $p^s$ can be efficiently computed using the \emph{regula falsi} method, cf. \cite{AC_transmission}.

We introduce the notation
$$\Gamma_{\rm{discr}}=\Bigl(\{1,\ldots,m\}\times \{\ell\in\eZ,\;\ell\leq -1\}\Bigr)\bigcup \Bigl(\{m+1,\ldots,m+n\}\times\{\ell\in\eZ,\;\ell\geq 0\}\Bigr),$$
for the set of all degrees of freedom at a fixed time step; the notation
$$
U^s= \Bigl( u^{h,s}_{\ell+\frac 12}\Bigr)_{(h,\ell)\in \Gamma_{\rm{discr}}}
$$
for the set of all the unknowns of the scheme at time step $s$; and the notation $\mathcal{S}^{\Delta x}\vec{u_0}$ to indicate the piecewise constant function corresponding to the discrete solution $\left( u_{\ell+\frac{1}{2}}^{h, s} \right)_{h,\ell,s}=(U^s)_s$: $\mathcal{S}^{\Delta x}\vec{u_0}=(u^{1,\Delta x},\ldots,u^{m,\Delta x},u^{m+1,\Delta x},\ldots,u^{m+n,\Delta x})$ where
\begin{equation}
  \begin{aligned}
 & u^{i,\Delta x}= \sum_{s\in\eN, \: \ell\leq -1} u_{\ell+\frac{1}{2}}^{i, s} \mathds{1}_{\Omega_i\cap ]\ell, \ell+1[}(x)\mathds{1}_{]s, s+1[}(t),&\quad i\in\{1,\ldots,m\},\\
 & u^{j,\Delta x}=\sum_{s\in\eN, \: \ell\geq 0} u_{\ell+\frac{1}{2}}^{j, s} \mathds{1}_{\Omega_j\cap ]\ell, \ell+1[}(x)\mathds{1}_{]s, s+1[}(t),&\quad  j\in\{m+1,\ldots,m+n\}.
  \end{aligned}
\end{equation}
\begin{lemma}\label{lem:well-balance}
The above finite volume numerical scheme is well-balanced, i.e., every element of $\mathcal{G}_{VV}$ corresponds to a stationary, constant per road solution of the scheme.
\end{lemma}
\begin{proof}
Let $\vec{k}\in\mathcal{G}_{VV} $ and denote by $\vec k^{\Delta x}$ the associated discrete function with entries $u^{h,\Delta x}\equiv k_h$, $h\in\{1,\ldots,m+n\}$. Then $\mathcal{S}^{\Delta x}\vec k^{\Delta x}=\vec k^{\Delta x}$, i.e., $\vec k^{\Delta x}$ is a stationary solution of the scheme.
Indeed, consider for instance an incoming road $i$.
Obviously, for all $\ell< -1$ the first iteration of \eqref{eq:main-scheme} initialized with constant initial values $u^{i,0}_{\ell+\frac 12}=k_i$
yields
$$u^{i,1}_{\ell+\frac 12}=k_i - \frac{\Delta t}{\Delta x} (G_i(k_i,k_i)-G_i(k_i,k_i))=k_i.$$
Moreover, due to \eqref{eq:p-step-scheme}, by definition of $G^*_i$ and of $\mathcal G_{VV}$ we have
$$
G_i(k_i,k_i)=f(k_i)=G(k_i,p_{\vec k})=G^*_i(\vec k),
$$
so that for  $\ell= -1$ we still find
$$u^i_{1,-\frac 12}=k_i - \frac{\Delta t}{\Delta x} (G_i(k_i,k_i)-G^*_i(\vec k))=k_i.$$
Similarly, we find that  \eqref{eq:p-step-scheme}, \eqref{eq:main-scheme} preserves the constant value $k_j$ on the $j$th outgoing road.
\end{proof}

\begin{remark}\label{rem:monotonicity-of-scheme}
  Due to the classical properties of Godunov fluxes recalled in \S~\ref{ssec:Godunov}, to the CFL condition~\ref{eq:CFL} and to Lemma \ref{pexists-Part2}(i) (see also Remark~\ref{rem:junctionGodunov}), the scheme is monotone in the following sense:
  \begin{equation}\label{eq:scheme-H}
   \forall s\in \eN\; \forall (h,\ell)\in\Gamma_{\rm{discr}}\colon\;\;\;\; u^{h,s+1}_{\ell+\frac 12}=H^{h}_{\ell+\frac 12}( U^{s} )
 \end{equation}
 for some functions $H^h_{\ell+\frac 12}$ that are monotone non-decreasing with respect to each of the arguments (actually each of these functions depends on at most $(m+n+1)$ entries of $U^{s}$).
 This implies in particular the order-preservation property
 \begin{equation}\label{eq:scheme-monotone}
   \forall (h,\ell)\in\Gamma_{\rm{discr}}\colon\;\;\;\; u^{h,s}_{\ell+\frac 12}\geq \hat u^{h,s}_{\ell+\frac 12} \;\;\; \Rightarrow \;\;\; \forall (h,\ell)\in\Gamma_{\rm{discr}}\colon\;\;\;\; u^{h,s+1}_{\ell+\frac 12}\geq \hat u^{h,s+1}_{\ell+\frac 12}.
 \end{equation}

 Since, moreover, the scheme is locally conservative by definition (due, in particular, to the condition \eqref{eq:p-step-scheme}),
 and because of the Lipschitz continuity of $G_h$ for all $h$,
  the scheme is conservative. It follows by the Crandall-Tartar Lemma (see, e.g., \cite{CrandallMajda}) that the scheme is $\Lp1$-contractive in the sense that the discrete analogue
 of \eqref{adaptedcontraction} (with $\vec \rho$ replaced by $\mathcal{S}^{\Delta x}\vec{u_0}$) and the initial condition $\vec u_0$ replaced by the discretized initial condition) holds true. We need a bit more specific property, which is the numerical counterpart of the Kato inequalities \eqref{eq:Kato}, in order to justify convergence to an admissible solution.
 \end{remark}

  Slightly extending the usual formalism (see \cite{EGHbook}), let $\top$, resp., $\bot$ denote (component per component, in the case of vector-valued arguments) the maximum, resp. the minimum operation on real scalars, vectors or sequences: e.g.,
  $$
  (k_1,k_2,k_3)\bot(\hat k_1,\hat k_2,\hat k_3)=(\min\{k_1,\hat k_1\},\min\{k_2,\hat k_2\},\min\{k_3,\hat k_3\}).
  $$
\begin{proposition}\label{prop:discreteKato}
 For any initial conditions $\vec{u_0},\hat {\vec {u}}_0 $ in $\Lp\infty(\Gamma; [0,R]^{m+n})$ the corresponding discrete solutions 
  of the scheme \eqref{eq:ICscheme}--\eqref{eq:schemefluxes} satisfy \emph{discrete Kato inequalities}. Namely,
 for
 $(h,\ell)\in \Gamma_{\rm{discr}}$, let $Q^{h}_\ell[U^s,\hat U^s]$ be defined by
  \begin{align*}
   & Q^{h}_\ell[U^s,\hat U^s]  = G_h\Bigl(u^{h,s}_{\ell-\frac 12}\top \hat u^{h,s}_{\ell-\frac 12}, u^{h,s}_{\ell+\frac 12}\top \hat u^{h,s}_{\ell+\frac 12}\Bigr) -
     G_h\Bigl(u^{h,s}_{\ell-\frac 12}\bot \hat u^{h,s}_{\ell-\frac 12}, u^{h,s}_{\ell+\frac 12}\bot \hat u^{h,s}_{\ell+\frac 12}\Bigr),\;\;
      \ell\neq 0,\\
   & Q^{h}_0[U^s,\hat U^s] =  G^*_h(\vec u^{\,*,s}\top \hat{\vec u}^{\,*,s})-G^*_h(\vec u^{\,*,s}\bot \hat{\vec u}^{\,*,s}).
  \end{align*}
  Then for all $\xi\in \mathcal D(]0,+\infty[\times \eR)$ such that $\xi \geq 0$ and $\partial_x \xi=0$ on $[-\Delta x/2,\Delta x/2]$,
  setting $\xi^s_{\ell+\frac 12}:=\xi(s\Delta t,(\ell+\frac 12)\Delta x)$ we have
 \begin{align*}
  -& \sum_{i=1}^m \sum_{s=1}^{+\infty} \Delta t \sum_{\ell\leq -1} \Delta x |u^{i,s}_{\ell+\frac 12}-
   \hat u^{i,s}_{\ell+\frac 12}| \frac{\xi^{s+1}_{\ell+\frac 12}- \xi^{s}_{\ell+\frac 12}}{\Delta t}\\
& - \sum_{i=1}^m \sum_{s=1}^{+\infty} \Delta t \sum_{\ell\leq -1} \Delta x Q^{i}_\ell[U^s,\hat U^s] \frac{\xi^{s+1}_{\ell+\frac 12}- \xi^{s+1}_{\ell-\frac 12}}{\Delta x}\\
  &-\sum_{j=m+1}^{m+n} \sum_{s=1}^{+\infty} \Delta t \sum_{\ell\geq 0} \Delta x |u^{j,s}_{\ell+\frac 12}-
  \hat u^{j,s}_{\ell+\frac 12}| \frac{\xi^{s+1}_{\ell+\frac 12}- \xi^{s}_{\ell+\frac 12}}{\Delta t}\\
&  - \sum_{j=m+1}^{m+n} \sum_{s=1}^{+\infty} \Delta t \sum_{\ell\geq 1} \Delta x Q^{j}_\ell[U^s,\hat U^s] \frac{\xi^{s+1}_{\ell+\frac 12}- \xi^{s+1}_{\ell-\frac 12}}{\Delta x} \leq 0.
 \end{align*}
\end{proposition}
Observe that in Proposition~\ref{prop:discreteKato}, we limit our attention to test functions constant in a neighbourhood of the junction. Thanks to this precaution, borrowed from \cite{AC_transmission}, and to the conservativity of the  Riemann solver at the junction the junction $\ell=0$
 does not contribute to the ``$Q[U,\hat U]\partial_x \xi$'' term of the discrete Kato inequality.
\begin{proof}
The argument is essentially classical in the context of monotone finite volume schemes.

 First, we state the ``per cell contraction principle'': for all $(h,\ell)\in \Gamma_{\rm{discr}}$
 \begin{equation}\label{eq:percell-contraction}
   \frac{|u^{h,s+1}_{\ell+\frac 12}- \hat u^{h,s+1}_{\ell+\frac 12}| - |u^{h,s}_{\ell+\frac 12}- \hat u^{h,s}_{\ell+\frac 12}| }{\Delta t} + \frac{Q^{h}_{\ell+1}[U^s,\hat U^s]-Q^{h}_\ell[U^s,\hat U^s]}{\Delta x} \leq 0,
 \end{equation}
 which readily follows from the observation that for all $a,b\in\eR$, $|a-b|=a\top b - a\bot b$ and from the monotonicity of $H^h_{\ell+\frac 12}$ in \eqref{eq:scheme-H}:
 \begin{align*}
& |u^{h,s+1}_{\ell+\frac 12}- \hat u^{h,s+1}_{\ell+\frac 12}|= u^{h,s+1}_{\ell+\frac 12} \top \hat u^{h,s+1}_{\ell+\frac 12}- u^{h,s+1}_{\ell+\frac 12} \bot \hat u^{h,s+1}_{\ell+\frac 12},\\
& u^{h,s+1}_{\ell+\frac 12} \top \hat u^{h,s+1}_{\ell+\frac 12}= H^{h}_{\ell+\frac 12}( U^{s} )\top H^{h}_{\ell+\frac 12}( \hat U^{s} )\leq H^{h}_{\ell+\frac 12}( U^{s}\top \hat U^{s} ),\\
&  u^{h,s+1}_{\ell+\frac 12} \bot \hat u^{h,s+1}_{\ell+\frac 12}= H^{h}_{\ell+\frac 12}( U^{s} )\bot H^{h}_{\ell+\frac 12}( \hat U^{s} )\geq H^{h}_{\ell+\frac 12}( U^{s}\bot \hat U^{s} ),\\
&  H^{h}_{\ell+\frac 12}( U^{s}\top \hat U^{s} )-H^{h}_{\ell+\frac 12}( U^{s}\bot \hat U^{s} )=|u^{h,s}_{\ell+\frac 12}- \hat u^{h,s}_{\ell+\frac 12}| - \frac{\Delta t}{\Delta x}\Bigl(Q^h_{\ell+1}[U^s,\hat U^s]-Q^h_{\ell}[U^s,\hat U^s] \Bigr),
 \end{align*}
 where the formula \eqref{eq:main-scheme} and the definition of $Q^h_{\ell}[U^s,\hat U^s]$ are used in the last line to express the function $H^h_{\ell+\frac 12}$.

 Second, observe that for all $\vec k,\hat{\vec k}$ we have by the conservativity property underlying the definition of $\vec G^{*}$ (see Remark~\ref{rem:junctionGodunov}):
 $$
 \sum_{i=1}^m G^{*}_h(\vec k\top \hat{\vec k}) - \sum_{j=m+1}^{m+n} G^{*}_h(\vec k\top \hat{\vec k})=0=
  \sum_{i=1}^m G^{*}_h(\vec k\bot \hat{\vec k}) - \sum_{j=m+1}^{m+n} G^{*}_h(\vec k\bot \hat{\vec k}).
 $$
 It remains to multiply the $(h,\ell)$'s inequality in \eqref{eq:percell-contraction} by the nonnegative quantity $\Delta t \Delta x\xi^{s+1}_{\ell+\frac 12}$ and sum up; the sum is finite since the support of $\xi$ is compact.
  Paying attention to the fact that
 $\xi^{s}_{-\frac 12}=\xi(s\Delta t,0)=\xi^{s}_{\frac 12}$,  from the definition of $Q^{h}_0[U^s,\hat U^s]$ we see that for all $s$,
 \begin{equation}\label{eq:entropyflux-conservativity}
  \sum_{i=1}^{m} \Delta x Q^{i}_0[U^s,\hat U^s] \xi^{s+1}_{-\frac 12}-
\sum_{j=m+1}^{m+n} \Delta x Q^{j}_0[U^s,\hat U^s] \xi^{s+1}_{\frac 12}=0.
\end{equation}
 Using the Abel transformation (discrete summation by parts) on each road with respect to the time superscripts $s$ and to the space subscripts $\ell$, bearing in mind
  \eqref{eq:entropyflux-conservativity} and the fact that for all $\ell\in \eZ$, $\xi^0_{\ell+\frac 12}=0$ by the choice of $\xi$, we derive the required discrete Kato inequality.
\end{proof}
\begin{theorem}\label{thm:convergence_scheme}
  Given an initial datum $\vec {u_0} \in \Lp\infty(\Gamma; [0,R]^{m+n})$, the numerical scheme \eqref{eq:ICscheme}--\eqref{eq:schemefluxes} converges to the unique admissible (in the sense of the equivalent Definitions~\ref{defBLN},~\ref{Gentropysolution},~\ref{adaptedentropysolution}) solution $\vec\rho$ of \eqref{eq:basic}, namely, $\mathcal{S}^{\Delta x}\vec{u_0}\to \vec \rho$ as $\Delta x=0$, subject to the CFL restriction \eqref{eq:CFL} on $\Delta t$.
 This ensures, in particular, existence of an admissible solution of \eqref{eq:basic} for every $\Lp\infty(\Gamma; [0,R]^{m+n})$ initial datum.
\end{theorem}
\begin{proof}
The proof follows the lines of the proof given in \cite{AC_transmission} (see also \cite{germes}). Let us only provide a sketch of the key arguments.

\noindent$\bullet$
  We start with compactly supported $\BV$ initial data. The compactness, in the sense of the a.e. convergence, of the family $\Bigl(\mathcal{S}^{\Delta x}\vec{u_0}\Bigr)_{\Delta x\in (0,1)}$ of discrete solutions is obtained with the $\BV_{loc}$ technique introduced in \cite{BGKT08-JEM}. It relies upon the monotonicity and the Crandall-Tartar lemma (see Remark~\ref{rem:monotonicity-of-scheme}).

  The compactness permits to define $\vec \rho$ as a limit of some sequence of discrete solutions $\mathcal{S}^{\Delta x_r}\vec{u_0}$, $\Delta_r\to 0$. At the end of the proof, having proved convergence to the unique admissible solution with datum $\vec {u_0}$, by the classical argument we can bypass the extraction of a sequence and get convergence of  $\mathcal{S}^{\Delta x}\vec{u_0}$, $\Delta x\to 0$.

\noindent$\bullet$ It is classical (see \cite{EGHbook}) to derive that for every $h\in\{1,\ldots,m+n\}$, $\rho_h$ fulfills the first property required in Definitions~\ref{defBLN},~\ref{Gentropysolution},~\ref{adaptedentropysolution}, namely, the Kruzhkov entropy inequalities \eqref{entropycondition} hold. In order to justify that $\vec\rho$ is an admissible solution, we just need to assess the second property in Definition~\ref{adaptedentropysolution}, i.e., the adapted entropy inequalities \eqref{eq:adp} that involve the junction.

\noindent$\bullet$ The main ingredient of the proof is the discrete Kato inequality proved in Proposition~\ref{prop:discreteKato}, where we choose $\hat {\vec {u}}_0 =\vec k$ with $\vec k\in \mathcal G_{VV}$. Observe that by Lemma~\ref{lem:well-balance}, we have  $\mathcal{S}^{\Delta x_r}\hat {\vec {u}}_0\equiv \vec k$; passing to the limit as $\Delta_r\to 0$, we will indeed derive the adapted entropy inequalities \eqref{eq:adp} for $\rho$ and complete the proof for compactly supported data of bounded variation.

  Following \cite{AC_transmission}, observe that test functions whose $x$-derivative vanishes near $x=0$ are dense (e.g., in the $\Cp1$ topology) in $\mathcal D(]0,+\infty[\times \R)$. Starting with the inequalities of Proposition~\ref{prop:discreteKato},
  with the same arguments as in the previous step (see \cite{EGHbook,AC_transmission}) we pass to the limit and find \eqref{eq:adp} first for such specific test functions $\xi$, and then (by density) for general test functions.

\noindent$\bullet$ Finally, as in \cite{germes,AC_transmission}, in two steps we extend the convergence result to general $\Lp\infty$ data $\vec{u_0}$. Extension
  to $\Lp1\cap \Lp\infty$ data follows by the density of compactly supported $\BV$ data in $\Lp1$ topology, with the help of $\Lp1$ contractivity of both the admissible solution semigroup and the discrete solutions semigroups. Extension to $\Lp\infty$ data is due to the property of finite domain of dependence
  \eqref{eq:L1loc-contraction} and to its discrete counterpart that follows from \eqref{eq:main-scheme} and from the CFL condition.
\end{proof}

\section{Vanishing viscosity limits are admissible solutions of \eqref{eq:basic}}\label{sec:viscosity}
This section is devoted to the proof that the solutions obtained as limit of vanishing viscosity approximations are
admissible solutions in the sense of Definitions~\ref{defBLN},~\ref{Gentropysolution},~\ref{adaptedentropysolution}, and are therefore unique.
The result follows from the combination of two ingredients: the construction of suitably many vanishing viscosity profiles, and the $\Lp1$ contraction
property known for the vanishing viscosity approximation. This ensures that vanishing viscosity solutions satisfy a family of adapted entropy inequalities
which is sufficiently large to fit the last claim of Theorem~\ref{thm:TFAE}.

\begin{proposition}\label{steadyviscprofiles}
    For any $\vec k$ in $\mathcal{G}_{VV}^o$,
    there exists  $p\in[0,R]$ and  $\vec{\rho^\eps} = (\rho^\eps_1, \ldots, \rho^\eps_{m+n})$  in $\Lp\infty(
    \Gamma; [0,R]^{m+n})$ such that 
    for all $h\in\{1,\ldots,m+n\}$, $\rho^\eps_h$ solves the ODE problem
    \begin{equation}\label{eq:steadyvisc}
      \begin{cases}
        f_h(\rho^\eps_h)_x = \eps (\rho^\eps_{h})_{xx},\quad \text{ in } \:\Omega_h\\
        \rho^\eps_h(0) = p,\\
        \lim_{x\in\Omega_h,\,|x|\to +\infty} \rho^\eps_h(x) = k_h.
      \end{cases}
    \end{equation}
\end{proposition}
\begin{proof}
  We take the value $p\in[0,R]$ that ensures that $\vec k\in \mathcal{G}_{VV}^o$, according to the definition \eqref{eq:GVV-o} of $\mathcal{G}_{VV}^o$.
  We consider here the case $h\leq m$, the other case being analogous. Let us integrate both sides of \eqref{eq:steadyvisc} on $]-\infty, x]$
  \begin{equation}\label{eq:int1}
    \eps(\rho^\eps_{i})_x (x) = F_i(\rho^\eps_i(x)),\;\; \text{where}\; F_i:\rho\in [0,R] \mapsto f_i (\rho) -f_i(k_i),
  \end{equation}
  being understood that one should have $(\rho^\eps_i)_x\to 0$ as $x\to -\infty$.
  Observe that by the definition of $\mathcal G_{VV}^o$, $k_i$ is the only zero of $F_i$ on $I[k_i,p]$. Assume, for the sake of being definite, that $k_i<p$ (the case $k_i>p$ is analogous, while in the case $k_i=p$ we have the obvious solution $\rho^\eps_i=k$).
  Then, the map
  $$
  P: ]k_i,p] \to ]-\infty,0],\;\; r\mapsto - \int_{r}^p \frac{\eps}{f_i (s) -f_i(k_i)}\,ds
  $$
  is well defined and strictly increasing, i.e., it admits the inverse $P^{-1}$ satisfying $P^{-1}(0)=p$, $\lim_{x\to -\infty} P^{-1}(x)=k_i$,
  $\lim_{x\to -\infty} (P^{-1})'(x)=0$. We find that $\rho^\eps_i$ solves \eqref{eq:int1} if and only if
\begin{equation}
 -x= \int_x^0 \frac{\eps\rho^\eps_{i,x}  }{f_i (\rho^\eps_i) -f_i(k_i)}\,dx =
 -P(\rho^\eps_i),
\end{equation}
i.e. $\rho^\eps_i=P^{-1}(x)$ is the required solution.
\end{proof}

Following \cite{coclitegaravello}, we consider the following parabolic regularization of the initial boundary value problem \eqref{eq:basic}
\begin{equation}
  \label{eq:CLv}
  \begin{cases}
    \rho^\eps_{h,t} + f(\rho^\eps_{h})_x = \eps \rho^\eps_{h,xx}, &\quad     t>0, \, x\in\Omega_h,\,  h \in \{1,...,m+n\},\\
    \rho^\eps_{h}(t,0) = \rho^\eps_{h'}(t,0), &\quad t>0,\, h,\,h'\in\{1,...,m+n\},\\
    \sum \limits_{i=1}^m \left( f(\rie(t,0)) - \eps \rho^\eps_{i,x}(t,0) \right) &\quad
    {} \\
    = \sum \limits_{j=m+1}^{m+n} \left( f(\rje(t,0)) - \eps \rho^\eps_{j,x}(t,0)     \right), &\quad t>0,\\
    \rho^\eps_{h}(0,x) = u_{h,\eps}^0(x), &\quad x\in\Omega_h, \,  h \in \{1,...,m+n\},
  \end{cases}
\end{equation}
where $\eps>0$.
Note that, in the spirit of \eqref{conservation}, the third and fourth lines of \eqref{eq:CLv} give the mass conservation
at the junction, since the sum of the incoming parabolic fluxes
is equal to the sum of the outgoing parabolic ones.
On the approximated initial conditions we assume that
\begin{equation}
  \label{eq:assiniteps}
  \begin{split}
    u^0_{h,\eps}\in \Wp{2,1}(\Omega_h)\cap \Cp\infty(\Omega_h),&\qquad     0\le u^0_{h,\eps}\le R,\\
    u^0_{h,\eps}\longrightarrow u^0_{h},&\qquad     \text{a.e. and in $\Lp{\text{\rm $p$}}(\Omega_h)$, $1\le p<\infty$},\quad\text{as $\eps\to0$},\\
    \norm{u^0_{h,\eps}}_{\Lp1(\Omega_h)}\le \norm{u^0_{h}}_{\Lp1(\Omega_h)},&\qquad
    \norm{(u^0_{h,\eps})_x}_{\Lp1(\Omega_h)}\le TV(u^0_{h}),\qquad \eps\norm{(u^0_{h,\eps})_{xx}}_{\Lp1(\Omega_h)}\le C_0,
  \end{split}
\end{equation}
for each $\eps>0,\,h\in\{1,...,m+n\}$, where $C_0$ is a positive constant independent on $\eps,\,h$.

First, notice that profiles constructed in Proposition~\ref{steadyviscprofiles} are solutions of \eqref{eq:CLv}. The obvious scaling property of these profiles, that we will denote $\vec k^\eps$ in the sequel, ensures the convergence of $\vec k^\eps(x)=\vec k^1(\frac x\eps)\to \vec k$ as $\eps \to 0$, for all $x\neq 0$. This readily yields a wide family of vanishing viscosity limits.
\begin{corollary}\label{cor:profiles}
  Any $\vec k\in \mathcal{G}_{VV}^o$ can be obtained as the limit in the $\Lploc1$ sense, as $\eps\to 0$, of a family $\vec k^\eps$ of stationary solutions of \eqref{eq:CLv}.
\end{corollary}

In general, using the theory of semigroups the authors of~\cite{coclitegaravello} proved the existence of  a unique solution
  $\vec{\rho^\eps}$  of~\eqref{eq:CLv} such that
  \begin{equation}\label{eq:regularity}
    \rho_h^\eps\in \Cp{}([0,\infty[;\Lp2(\Omega_h))    \cap \Lploc1(]0,+\infty[; \Wp{2,1}(\Omega_h)),    \qquad\eps>0,\quad h\in \{1,\ldots,m+n\},
  \end{equation}
  in particular
  \begin{equation}\label{eq:regularity-bis}
  (\rho_h^\eps)_t\in \Lploc1(]0,+\infty[,\Lp1(\Omega_h)), \qquad\eps>0,\quad h\in \{1,\ldots,m+n\}.
  \end{equation}
  Moreover, if we have two different initial conditions
  \begin{equation*}
    (\widetilde u_{1,\eps}^0,\ldots, \widetilde u_{m+n,\eps}^0),
    \qquad  (u_{1,\eps}^0,\ldots, u_{m+n,\eps}^0)
  \end{equation*}
  for~(\ref{eq:CLv}) satisfying~(\ref{eq:assiniteps}), then
  the corresponding solutions to~(\ref{eq:CLv})
  \begin{equation*}
    (\widetilde\rho_1^\eps,\ldots, \widetilde\rho_{m+n}^\eps),
    \qquad    (\rho_1^\eps,\ldots, \rho_{m+n}^\eps)
  \end{equation*}
 are stable in the following sense
  \begin{equation}
    \label{eq:stability}
    \begin{split}
      \sum\limits_{h=1}^{m+n}&\norm{\rhe(t,\cdot)-\widetilde{\rho}^\eps(t,\cdot)}_{\Lp1(\Omega_h)}
      \le \sum \limits_{h=1}^{m+n}
      \norm{u_{h,\eps}^0-\widetilde u^0_{h,\eps}}_{\Lp1(\Omega_h)},
    \end{split}
  \end{equation}
for every $t \ge 0$.

The compactness argument of \cite{coclitegaravello} is based on the compensated compactness theory \cite{TartarI} and the following
a priori estimates
\begin{align}
\label{eq:linfty}
&0\le \rho^\eps_{h}\le R,\qquad h\in\{1,\ldots,m+n\},\\
\label{eq:l1}
&\sum\limits_{h=1}^{m+n}\norm{\rhe(t,\cdot)}_{\Lp1(\Omega_h)}\le\sum\limits_{h=1}^{m+n}\norm{u^0_{h}}_{\Lp1(\Omega_h)},\\
\label{eq:l2}
&\sum\limits_{h=1}^{m+n}\norm{\rhe(t,\cdot)}_{\Lp2(\Omega_h)}^2
+2\eps\int_0^t\left(\sum\limits_{h=1}^{m+n} \norm{\rho^\eps_{h,x}(s,\cdot)}_{\Lp2(\Omega_h)}^2\right)ds\\
&\notag\qquad\qquad\qquad\le\sum\limits_{h=1}^{m+n}\norm{u_{h,\eps}^0}_{\Lp2(\Omega_h)}^2+2|n-m|\max_h\norm{f_h}_{\Wp{1,\infty}(0,R)}t,\\
\label{eq:l1t}
&\sum\limits_{h=1}^{m+n}\norm{\rho^\eps_{h,t}(t,\cdot)}_{\Lp1(\Omega_h)}\le(m+n)C_0+\max_h\norm{f_h'}_{\Lp\infty(0,R)}\sum\limits_{h=1}^{m+n}TV(u^0_{h}),
 \end{align}
for every $t\ge0$ and $\eps>0$.

The main result in \cite{coclitegaravello} shows that there exist a sequence $\{\eps_\ell\}_{\ell\in\N}\subset(0,\infty)$, $\eps_\ell\to 0$
  and  a solution $\vec{\rho}$   of~\eqref{eq:basic}, in the sense of
  Definition~\ref{solutionjunction}, such that
    \begin{align}
    \label{eq:comp3-th}
    & \rho^{\eps_\ell}_h \longrightarrow \rho_h,\qquad \textrm{a.e. and in } \Lploc{p}(\eR_+\times\Omega_h),\,1\le p<\infty,
  \end{align}
for every $h\in\{1,\ldots,m+n\}$, where $\vec{\rho^{\eps_\ell}}$ is the corresponding solution of \eqref{eq:CLv}.

\begin{remark}\label{CG-extended}
  Actually these results were proved in \cite{coclitegaravello} in the case there all the functions $f_h$ coincide, $h=1,\ldots,m+n$, and moreover the strict concavity of the flux function is assumed. Extension to different fluxes $f_h$ on different roads is straightforward. The strict concavity assumption can be replaced by the nonlinearity assumption $(NLD)$: e.g., the strong precompactness result of \cite{Panov-compactness} can be used on each road in the place of the compensated compactness method.
\end{remark}

Here we improve the result of \cite{coclitegaravello} showing the following:

\begin{theorem}\label{VVsolareGentropysol}
 Assume  \eqref{eq:assiniteps}. Let $\{\vec{\rho^\eps}\}_{\eps>0}$ be the family of solutions of \eqref{eq:CLv}.
 We have that
\begin{equation}
\label{eq:comp3-th.1}
\rho^{\eps}_h \longrightarrow \rho_h,\qquad \textrm{a.e. and in }\Lploc{\text{\rm $p$}}(\R_+\times\Omega_h),\,1\le p<\infty,
 \end{equation}
where $\vec{\rho}$ is the unique
admissible solution of \eqref{eq:basic} in the sense of Definitions~\ref{defBLN},~\ref{Gentropysolution},~\ref{adaptedentropysolution}.
\end{theorem}

\begin{proof}
Let $\{\vec{\rho\,}^{\eps_\ell}\}_{\ell\in\N}$ be the solutions of \eqref{eq:CLv} converging to $\vec{\rho}$ as in \eqref{eq:comp3-th}.
 According to Theorem~\ref{thm:TFAE}, it is enough to justify that $\vec \rho$ satisfies the per road Kruzhkov inequalities \eqref{entropycondition}
for all $k\in [0,R]$ and the adapted entropy inequalities \eqref{eq:adp} for all $\vec k\in \mathcal G_{VV}^o$. The first claim is classical, see, e.g., \cite{coclitegaravello}. We only need to justify the second claim.

According to Corollary~\ref{cor:profiles}, given $\vec k \in \mathcal{G}_{VV}^o$ there exist $\vec k^\eps$ stationary viscosity profiles converging to $\vec k$, as $\eps\to 0$. Now, arguing as in \cite[p.\,1773]{coclitegaravello} but inserting in addition a non-negative test function $\xi\in\mathcal{D}(]0,+\infty[\times\eR)$, we find the following Kato inequality:
 \begin{equation}
 \label{eq:adp-epsilon}
     \sum_{h=1}^{m+n}\left(\int_{\eR_+}\int_{\Omega_h} \left\{|\rho^{\eps_\ell}_h - k^{\eps_\ell}_h|\xi_{t} +q_h(\rho^{\eps_\ell}_h ,k^{\eps_\ell}_h)\xi_{x}+{\eps_\ell}|\rho_h^{\eps_\ell} - k^{{\eps_\ell}}_h|_{x}\xi_{x} \right\}\,dx\,dt \right)
     \geq 0.
 \end{equation}
Let us give the details of the calculation. Bearing in mind the regularity \eqref{eq:regularity},\eqref{eq:regularity-bis} of solutions and
the Lipschitz regularity of $f_h$, it is a classical matter to obtain the per road Kato inequalities:
\begin{equation}
\label{eq:GMCbasic}
\int_{\eR_+}\!\!\int_{\Omega_h}\left(|\rho_h^{\eps_\ell} - k^{{\eps_\ell}}_h|\xi_{t}
+q_h(\rho^{\eps_\ell}_h ,k^{{\eps_\ell}}_h)\xi_{x}+{\eps_\ell}|\rho_h^{\eps_\ell} - k^{{\eps_\ell}}_h|_{x}\xi_{x}\right) \,dx\,dt \ge 0,
\end{equation}
$h\in\{1,\ldots,m+n\}$, for all $\xi\in \mathcal{D}(]0,+\infty[\times(\eR\setminus\{0\}))$. Moreover, bearing in mind the existence
of strong traces of $q_h(\rho^{\eps_\ell}_h ,k^{{\eps_\ell}}_h)+{\eps_\ell}|\rho_h^{\eps_\ell} - k^{{\eps_\ell}}_h|_{x}$ as $x\to 0$,
using the truncations defined by \eqref{eq:truncation-test}, with the same argument as in the proof of Proposition~\ref{thm:uniqueness}
we can generalize \eqref{eq:GMCbasic} to test functions $\xi$ not necessarily vanishing near the junction; the appropriate boundary terms appear. Summing up the resulting inequalities, and using the conservativity conditions contained in \eqref{eq:CLv} both for solutions $\vec{\rho\,}^{\eps_\ell}$ and $\vec{k}^{\eps_\ell}$, we find\begin{align*}
0\ge
&-\sum_{h=1}^{m+n}\int_{\eR_+}\!\!\int_{\Omega_h}\left(|\rho_h^{\eps_\ell} - k^{{\eps_\ell}}_h|\xi_{t}
+q_h(\rho^{\eps_\ell}_h ,k^{{\eps_\ell}}_h)\xi_{x}+{\eps_\ell}|\rho_h^{\eps_\ell} - k^{{\eps_\ell}}_h|_{x}\xi_{x}\right) \,dx\,dt\\ 
&+\sum_{h=1}^{m}\int_{\eR_+}\left(q_h(\rho^{\eps_\ell}_h(t,0) ,k^{{\eps_\ell}}_h)\xi(t,0)-{\eps_\ell}|\rho_h^{\eps_\ell}(t,0) - k^{{\eps_\ell}}_h|_{x}\xi(t,0)\right) \,dt\\
&-\sum_{h=m+1}^{m+n}\int_{\eR_+}\left(q_h(\rho^{\eps_\ell}_h(t,0) ,k^{{\eps_\ell}}_h)\xi(t,0)-{\eps_\ell}|\rho_h^{\eps_\ell}(t,0) - k^{{\eps_\ell}}_h|_{x}\xi(t,0)\right) \,dt\\
=&-\sum_{h=1}^{m+n}\int_{\eR_+}\!\!\int_{\Omega_h}\left(|\rho_h^{\eps_\ell} - k^{{\eps_\ell}}_h|\xi_{t}
+q_h(\rho^{\eps_\ell}_h ,k^{{\eps_\ell}}_h)\xi_{x}+{\eps_\ell}|\rho_h^{\eps_\ell} - k^{{\eps_\ell}}_h|_{x}\xi_{x}\right) \,dx\,dt\\
&+\int_{\eR_+} \sgn{\rho_h^{\eps_\ell}(t,0) - k^{{\eps_\ell}}_h}   \left(
\sum_{i=1}^{m}(f_i(\rho^{\eps_\ell}_i(t,0)) -{\eps_\ell}\rho_{i,x}^{\eps_\ell}(t,0))\right.\\
&\qquad\left.-\sum_{j=m+1}^{m+n}(f_j(\rho^{\eps_\ell}_j(t,0)) -{\eps_\ell}\rho_{j,x}^{\eps_\ell}(t,0))\right) \xi(t,0)\,dt\\
&-   \left(\sum_{i=1}^{m}f_i(k^{{\eps_\ell}}_h) -\sum_{j=m+1}^{m+n}f_j(k^{\eps_\ell}_j)\right) \int_{\eR_+} \sgn{\rho_h^{\eps_\ell}(t,0) - k^{{\eps_\ell}}_h}\xi(t,0)\,dt\\
=&-\sum_{h=1}^{m+n}\int_{\eR_+}\!\!\int_{\Omega_h}\left(|\rho_h^{\eps_\ell} - k^{{\eps_\ell}}_h|\xi_{t}
+q_h(\rho^{\eps_\ell}_h ,k^{{\eps_\ell}}_h)\xi_{x}+{\eps_\ell}|\rho_h^{\eps_\ell} - k^{{\eps_\ell}}_h|_{x}\xi_{x}\right) \,dx\,dt.
 \end{align*}

 Passing to the limit in \eqref{eq:adp-epsilon} as $\eps_\ell\to 0$, keeping in mind the second term of estimate \eqref{eq:l2}, we find inequality \eqref{eq:adp}. This concludes the proof.
\end{proof}

\bibliographystyle{acm}
    \bibliography{biblio-ACDjunction}

\end{document}